\documentclass[sn-mathphys,Numbered]{sn-jnl}

\usepackage{natbib}
\usepackage{graphicx}%
\usepackage{multirow}%
\usepackage{amsmath,amssymb,amsfonts}%
\usepackage{amsthm}%
\usepackage{mathrsfs}%
\usepackage[title]{appendix}%
\usepackage{xcolor}%
\usepackage{textcomp}%
\usepackage{manyfoot}%
\usepackage{booktabs}%
\usepackage{algorithm}%
\usepackage{algorithmicx}%
\usepackage{algpseudocode}%
\usepackage{listings}%
\usepackage{amsfonts}
\usepackage{bbm}
\usepackage{mathrsfs}
\usepackage{lipsum}
\usepackage{geometry}
\usepackage{multirow}
\usepackage{graphics,graphicx,epsf,subfigure,epstopdf,epsfig}
\usepackage{boxedminipage}
\usepackage{caption}
\usepackage{enumerate}
\usepackage{enumitem}
\usepackage{accents}
\usepackage{fancybox}
\usepackage{slashbox}
\usepackage{bm}

\def\x{\mathbf{x}}
\def\e{\mathbf{e}}
\def\v{\mathbf{v}}
\def\w{\mathbf{w}}
\def\y{\mathbf{y}}
\def\z{\mathbf{z}}
\def\a{\mathbf{a}}
\def\b{\mathbf{b}}
\def\c{\mathbf{c}}

\def\u{\mathbf{u}}
\def\v{\mathbf{v}}
\def\0{\mathbf{0}}
\def\R{\mathbb{R}}

\def\X{\mathcal{X}}
\def\A{\mathcal{A}}

\def\Y{\mathcal{Y}}

\theoremstyle{thmstyleone}%
\newtheorem{theorem}{Theorem}[section]
\newtheorem{proposition}{Proposition}[section]%
\newtheorem{assumption}{Assumption}[section]%

\theoremstyle{thmstyletwo}%
\newtheorem{example}{Example}[section]%
\newtheorem{remark}{Remark}[section]%

\theoremstyle{thmstylethree}%
\newtheorem{definition}{Definition}[section]%
\numberwithin{equation}{section}
\begin{document}
\bibliographystyle{plainnat}
\title{Nonsmooth convex-concave saddle point problems with cardinality penalties}

\author{Wei Bian\footnote{\footnotesize School of Mathematics, Harbin Institute of Technology, Harbin, China. {Email: \texttt{bianweilvse520@163.com}}. The research of this author
is partially supported by National Natural Science Foundation of China grants (12271127, 62176073), National Key Research and Development
Program of China (2021YFA1003500) and Fundamental Research Funds
for Central Universities (2022FRFK060017).}
 \hspace{0.05in}   and \hspace{0.05in}   Xiaojun Chen\footnote{\footnotesize Department of Applied Mathematics, Hong Kong Polytechnic University,
    Hong Kong, China. {Email: \texttt{maxjchen@polyu.edu.hk}}. The research of this author
is partially supported by Hong Kong Research Grant Council project PolyU15300022. Corresponding author.\\  {\bf Submitted  22 October, 2023; revised 24 March, 2024}}
        }
%
%
%
%
%
%
%
%

\abstract{In this paper, we focus on a class of convexly constrained nonsmooth convex-concave saddle point problems with cardinality penalties. Although such nonsmooth nonconvex-nonconcave and discontinuous min-max problems may not have a saddle point, we show that they have a local saddle point and a global minimax point, and  some local saddle points have the lower bound properties. We define a class of strong local saddle points  based on the lower bound properties
for stability of variable selection. Moreover we give a framework to construct continuous relaxations of the discontinuous min-max problems based on the convolution, such that they have the same saddle points with the original problem. We also establish  the relations between the continuous relaxation problems and the original problems regarding local saddle points, global minimax points, local minimax points and stationary points.
Finally, we illustrate our results with distributionally robust sparse convex regression, sparse robust bond portfolio construction and sparse convex-concave logistic regression saddle point problems.}





\maketitle

\vspace{-0.15in}
\section{Introduction}\label{sec:intro}
Let $c:\R^n\times\R^m\rightarrow\R$ be a Lipschitz continuous function
with $c(\x,\y)$ {convex} in $\x\in \R^n$ for $\y\in\R^m$ and {concave} in $\y\in \R^m$ for $\x\in\R^n$, $g:\R^n\rightarrow \R^{\hat{n}}$ and $h:\R^m\rightarrow \R^{\hat{m}}$ be continuously differentiable functions. For a vector $\a\in\R^k$, $\|\a_+\|_0$ is
the cardinality function for the positive elements in $\a$, that is, $\|\a_+\|_0=\|\max\{\a,0\}\|_0=\sum_{i=1}^k(\max\{\a_i,0\})^0$ with $0^0=0$.
In this paper, we consider the saddle point problems with cardinality penalties in the following form:
\begin{equation}\label{obb}
\min_{\x\in\X}\max_{\y\in\Y} \,f(\x,\y):=c(\x,\y)+\lambda_1\|g(\x)_+\|_0
-\lambda_2\|h(\y)_+\|_0,
\end{equation}
where the feasible sets $\X\subset \R^n$ and $\Y\subset \R^m$ are convex and compact,
and the penalty parameters $\lambda_1, \lambda_2\in \R$ are positive.

In the last few years, the min-max problems have been found many interesting applications in machine learning and data science, especially the generative adversarial network (GAN) \cite{Goodfellow-2014,Goodfellow-2016,Jiang-Chen} and adversarial training \cite{Bai}. Problem \eqref{obb} is a discontinuous and  nonconvex-nonconcave min-max problem, i.e. $f$ is discontinuous in $\X \times \Y$,  $f(\cdot,\y)$ is not convex for some fixed $\y\in\Y$ and $f(\x,\cdot)$ is not concave for some fixed $\x\in\X$.
A special case of \eqref{obb} is
\begin{equation}\label{obb-2}
\min_{\x\in\X}\max_{\y\in\Y} \,c(\x,\y)+\lambda_1\|(\x-\underline{\a})_+\|_0+
\lambda_1\|(\overline{\a}-\x)_+\|_0
-\lambda_2\|(\y-\underline{\b})_+\|_0-\lambda_2\|(\overline{\b}-\y)_+\|_0,
\end{equation}
where $\underline{\a}$, $\overline{\a}\in\R^n$ and $\underline{\b}$, $\overline{\b}\in\R^m$. In particular, if $\underline{\a}=\overline{\a}=\underline{\b}=\overline{\b}=\0$, then \eqref{obb-2}
reduces to the {convex-concave} saddle point problem with $\ell_0$ penalties as follows
\begin{equation}\label{obb-4}
\min_{\x\in\X}\max_{\y\in\Y} \, c(\x,\y)+
\lambda_1\|\x\|_0-\lambda_2\|\y\|_0.
\end{equation}

In 1928, von Neumann \cite{von-Neumann-1928} proved that when $c$ is a bilinear function, and $\X$, $\Y$ are two finite dimensional simplices,
\begin{equation}\label{obb-c}
\min_{\x\in\X}\max_{\y\in\Y}\,c(\x,\y)
\end{equation}
has a saddle point and it holds
\begin{equation}\label{eq-cc}
\min_{\x\in\X}\max_{\y\in\Y}\,c(\x,\y)=\max_{\y\in\Y}\min_{\x\in\X}\,c(\x,\y).
\end{equation}
This pioneering
work has inspired a number of seminal contributions in the existence theory of saddle points of min-max problems in economics and engineering \cite{Pang-book2,Rock-Wets,Nikaido,Sion,Rock,Debreu}. In 1949, Shiffman \cite{Shiffman} gave a new proof of von Neumann's minimax theorem with a generalization to convex-concave functions. In 1952, Fan \cite{Fan1952} established the minimax theory for the convex-concave function that is semicontinuous in one of the two variables.
Based on Brouwer's fixed point theorem, Nikaido \cite{Nikaido} proved \eqref{eq-cc}
for the continuous and quasi-convex-concave function $c$. Here, we call $c$ is quasi-convex-concave if $c$ is quasi-convex
in $\x\in \R^n$ for $\y\in\R^m$ and {quasi-concave} in $\y\in \R^m$ for $\x\in\R^n$.
In 1958,
Sion \cite{Sion} generalized von Neumann's result, and showed that if $c$ is quasi-convex-concave and lower semicontinuous-upper semicontinuous, then \eqref{obb-c} has a nonempty saddle point set whose closedness and convexity were pointed out  in \cite{Rock}.
Moreover, from \cite[Theorem 1.4.1]{Pang-book2} in the book by Facchinei and Pang, we know that
 \begin{equation}\label{gap}
\min_{\x\in\X}\max_{\y\in\Y}\,f(\x,\y)=\max_{\y\in\Y}\min_{\x\in\X}\,f(\x,\y)
\end{equation}
 is a necessary and sufficient condition for the existence of a saddle point of $f$ over $\X\times \Y$.

When $f$ is not convex-concave, \eqref{gap} fails in general.
The concept of a local saddle point is defined by considering (\ref{gap}) locally
at a point in $\X\times \Y$.
However, a local saddle point also may not exist for nonconvex-nonconcave min-max problems in machine learning. In \cite{Jin-Netrapalli-Jordan}, Jin, Netrapalli and Jordan  gave the definitions of a global minimax point and a local minimax point  by considering the min-max problem as a two-player sequential game. Necessary and sufficient conditions for a local minimax points of min-max problems were studied in \cite{Jin-Netrapalli-Jordan,Dai,Guo,Jiang-Chen}.
Most recently, Chen and Kelley \cite{Chen-Kelley} showed that a min-max problem for robust linear least squares problems does
not have a saddle point, a local saddle point and a local minimax point, while it has infinitely first order
stationary points and finite global minimax points. However, the set of first order stationary points and the set of global minimax points do not have a common point.

The cardinality functions in problem (\ref{obb}) play important roles to ensure the sparsity of the desirable solutions and improve the estimation accuracy by selecting important feature parameters. In the last decades, sparse minimization models with cardinality regularization terms have been widely used for sparse signal recovery, sparse variable selection, compressed sensing and statistical learning
\cite{Beck2018-MP,BC-SINUM,Bruckstein,Cui-Pang-2022,Soubies2017}.
Advanced mathematical and statistical theory and efficient algorithms have been developed for sparse minimization \cite{BC-SINUM,Candes1,Chen-Xu-Ye,Jiao2015}.
Recently, He et al. \cite{CuiJS} systematically compared solutions of a special quadratic minimization problem with $\ell_0$ penalty, $\ell_1$ penalty and capped-$\ell_1$ penalty.
However, mathematical theory and numerical algorithms for sparse min-max problems with cardinality penalties have not been systematically studied. Moreover,
the existence of solutions and continuous relaxation methods to \eqref{obb} have not been carefully studied.

Approximating cardinality functions in optimization problems by continuous or smooth functions is a promising approach.
Many continuous relaxations to the cardinality function have been brought forward, such as the $\ell_1$ \cite{Candes1}, SCAD \cite{Fan1997}, hard thresholding \cite{Fan1997}, $\ell_p$ ($0<p<1$) \cite{Chen-Xu-Ye}, MCP \cite{Zhang2011}, capped-$\ell_1$ \cite{Peleg2008}, CEL0 \cite{Soubies2017}, etc.  In this paper, we
construct continuous approximations to the cardinality functions in \eqref{obb} based on the convolution \cite{Chen-Mangasarian,chen-MP}, which include most relaxation functions to the cardinality function.

The main contributions of this paper have four parts.
\begin{itemize}
\item We prove the existence of a local saddle point and a global minimax point of \eqref{obb}, and define a class of strong local saddle points that have some desirable sparse properties.
\item Based on the convolution, we introduce two classes of density functions to provide a unified method for constructing the continuous relaxations with different smoothness to the cardinality function,  which induce some popular continuous penalties in sparse optimization. Moreover, we propose the continuous relaxation problem of \eqref{obb}, which has both the local saddle points and global minimax points.
\item We establish the relations between \eqref{obb} and its continuous relaxations regarding the saddle points, local saddle points, local minimax points and global minimax points.
    Moreover, we define the first order and second order stationary points of the continuous relaxation problem, which not only correspond to the strong local saddle points of \eqref{obb}, but also have  promising computational tractability.
\item We  show the gradient consistency of a class of smoothing convex-concave functions of nonsmooth functions $c$.  Moreover we prove that any accumulation point of weak d(irectional)-stationary points of the smoothing relaxation problem  is a weak d-stationary point of the nonsmooth relaxation problem  as the smoothing parameter  goes to zero.
\end{itemize}

The rest of this paper is organized as follows. In Section \ref{section2}, we prove the existence of local saddle points and  minimax points of (\ref{obb}). In Section \ref{section3}, we construct the continuous relaxations to \eqref{obb}. In Section \ref{section4}, we establish the relation between (\ref{obb}) and its continuous relaxation problems. The smoothing functions of nonsmooth function  $c$ are studied at the end of this section. In Section \ref{section5}, we study the first order and second order stationary points of the continuous relaxation problems for a particular class of (\ref{obb}) and their relation with the
strong local saddle points of (\ref{obb}). In Section \ref{section6}, we show the applications of problem \eqref{obb}.

{\bf Notation}
Let $\mathbb{R}_+=[0,+\infty)$, $\mathbb{R}_{++}=(0,+\infty)$ and $[n]=\{1,2,\ldots,n\}$ for a positive integer $n$. For a matrix ${\bf A}\in\R^{n\times m}$, ${\bf A}_{ij}$ means the element of ${\bf A}$ at the $i$th row and $j$th column.
Let $\e_i$ be the vector with $1$ at the $i$th element and $0$ for the others and $\e$ be the vector with $1$ for all elements.
For a Lipschitz continuous function $c:\mathbb{R}^n\times\mathbb{R}^m\rightarrow\mathbb{R}$, $\partial_{\x}c(\bar{\x},\bar{\y})$ and $\partial_{\y}c(\bar{\x},\bar{\y})$ mean the Clarke subgradients of $c$ with respect to $\x$ and $\y$ at point $(\bar{\x},\bar{\y})$, respectively. {When $c$ is Lipschitz continuously differentiable, $\partial^2c(\bar{\x},\bar{\y})$ means the Clarke generalized Hessian of $c$ at point $(\bar{\x},\bar{\y})$, $\partial_{\x\x}^2c(\bar{\x},\bar{\y})$ and $\partial_{\y\y}^2c(\bar{\x},\bar{\y})$ mean the Clarke generalized Hessian of $c(\x,\bar{\y})$ with respect to $\x$ at point $\bar{\x}$ and $c(\bar{\x},{\y})$ with respect to $\y$ at point $\bar{\y}$, respectively. }
For $\x\in\mathbb{R}^n$ and $\delta>0$, $\mathbf{B}(\x,\delta)$ means the closed ball centered at $\x$ with radius $\delta$,
$\mathcal{A}^+(\x)=\{i\in[\hat{n}]:g_i(\x)>0\}$,
$\mathcal{A}^0(\x)=\{i\in[\hat{n}]:g_i(\x)=0\}$
 and
$\mathcal{A}^+_{\delta}(\x)=\{i\in[\hat{n}]:0<g_i(\x)<\delta\}$.
Similarly, denote
$\mathcal{B}^+(\y)=\{j\in[\hat{m}]:h_j(\y)>0\}$,
$\mathcal{B}^0(\y)=\{j\in[\hat{m}]:h_j(\y)=0\}$ and
$\mathcal{B}^+_{\delta}(\y)=\{j\in[\hat{m}]:0<h_j(\y)<\delta\}$.
For a set $S\subseteq\mathbb{R}^n$ and $i\in[n]$,
$S_i=\{\x_i:\x\in S\}$, $\|S\|_{\infty}=\sup\{\|\x\|_{\infty}:\x\in S\}$ and $\mbox{co}\{S\}=\{\lambda \x^1+(1-\lambda)\x^2:\x^1,\x^2\in S,\lambda\in[0,1]\}$. For sets $S,\bar{S}\subseteq\mathbb{R}^n$, $S+\bar{S}=\{\x^1+\x^2:\x^1\in S,\x^2\in \bar{S}\}$. For a closed convex subset $\Omega\subseteq\R^n$ and $\x\in\Omega$, $\mathcal{N}_{\Omega}(\x)$ means the normal cone to $\Omega$ at $\x$.

\section{Existence of local saddle points of problem \eqref{obb}}\label{section2}
In this section, we prove the existence of local saddle points and  minimax points of  problem \eqref{obb}.  We also define a class of strong local saddle points of \eqref{obb} and provide its relation with saddle points of problem (\ref{obb-c}) in  a certain subset of $\X\times \Y$. {First of all, we give some necessary definitions.}
\begin{definition}\label{defn1}
A point $(\x^*,\y^*)\in\X\times\Y$ is called a \textbf{saddle point} of problem \eqref{obb}, if for all $(\x,\y)\in\X\times\Y$, it holds
\begin{equation}\label{eq-N}
f(\x^*,\y)\leq f(\x^*,\y^*)\leq f(\x,\y^*).
\end{equation}
We call $(\x^*,\y^*)\in\X\times\Y$ a \textbf{local saddle point} of problem \eqref{obb}, if there exists a $\delta>0$ such that \eqref{eq-N} holds for all $\x\in\X\cap\mathbf{B}(\x^*,\delta)$ and $\y\in\Y\cap\mathbf{B}(\y^*,\delta)$.
\end{definition}

\begin{definition}\label{defn-min}
A point $(\x^*,\y^*)\in\X\times\Y$ is called a \textbf{{global minimax point}} of \eqref{obb}, if  for all $(\x,\y)\in\X\times\Y$, we have
$$
f(\x^*,\y)\leq f(\x^*,\y^*)\leq\max_{\y'\in \Y} f(\x,\y').$$
We call $(\x^*,\y^*)\in\X\times\Y$ a \textbf{{local minimax point}} of \eqref{obb}, if there exist a $\delta_0>0$ and a function $\pi:\R_+\rightarrow\R_+$ satisfying $\pi(\delta)\rightarrow0$ as $\delta\rightarrow0$ such that for any $\delta\in(0,\delta_0]$, $\x\in\X\cap\mathbf{B}(\x^*,\delta)$ and $\y\in\Y\cap\mathbf{B}(\y^*,\delta)$, it holds
$$
f(\x^*,\y)\leq f(\x^*,\y^*)\leq\max_{\y'\in\Y\cap\mathbf{B}(\y^*,\pi(\delta))}f(\x,\y').$$
 \end{definition}

It is easy to see that a local saddle point is a local minimax point. But a global minimax point is not necessarily a local minimax point.
The two inequalities in \eqref{eq-N} for $\x\in\X$ and $\y\in\Y$ can be equivalently expressed by
\begin{equation}\label{eq2.2}
\x^*\in\arg\min_{\x\in\mathcal{X}}f(\x,\y^*)\quad\mbox{and}
\quad\y^*\in\arg\max_{\y\in\mathcal{Y}}f(\x^*,\y),
\end{equation}
respectively.
Since $\X$ and $\Y$ are compact, the lower semicontinuity of $((\cdot)_+)^0$ guarantees the existence of the solutions to the two optimization problems in \eqref{eq2.2}, but $(\x^*,\y^*)$ may not be able to solve both the minimization and maximization simultaneously.
{This means that the saddle point set of problem (\ref{obb})
may be empty and
$\min_{\x\in\X}\max_{\y\in\Y}\,f(\x,\y)\neq\max_{\y\in\Y}\min_{\x\in\X}\,f(\x,\y)$ (see Example \ref{example-s}).

Note that a nonconvex-nonconcave function may not have a saddle point, a local saddle point, or even a local minimax point. Fortunately, we can prove the existence of global minimax points, local saddle points and local minimax points of problem \eqref{obb} without any other assumption.
\begin{proposition}\label{prop12}
Min-max problem \eqref{obb} always has a global minimax point.
\end{proposition}
\begin{proof}
By the compactness of $\Y$, we can define
$
p(\x)=\max_{\y\in\Y}f(\x,\y).$
Since function $f(\cdot,\y)$ in \eqref{obb} is lower semicontinuous for any fixed $\y\in\R^m$ and $\X$ is compact, $p$
is lower semicontinuous on $\X$.
Then, there exists a global solution to $\min_{\x\in\X}p(\x)$, denoted by $\x^*$, i.e.
\begin{equation}\label{eq-e9}
\max_{\y'\in\Y}f(\x^*,\y')\leq \max_{\y'\in\Y}f(\x,\y'), \quad\forall\, \x\in\X.
\end{equation}
The upper semicontinuity of $f(\x^*,\cdot)$ and the compactness of $\Y$ ensure the existence of the solution to $\max_{\y'\in\Y}f(\x^*,\y')$, denoted by $\y^*$,
which implies
\begin{equation}\label{eq-e10}
f(\x^*,\y^*)=\max_{\y'\in\Y}f(\x^*,\y')\geq f(\x^*,\y), \quad\forall \y\in\Y.
\end{equation}
Therefore, \eqref{eq-e9} together with \eqref{eq-e10} implies that
$(\x^*,\y^*)$ is a global minimax point of \eqref{obb}.
\end{proof}

\begin{proposition}\label{prop4}
Any saddle point of \eqref{obb-c} is a local saddle point of \eqref{obb}.
\end{proposition}
\begin{proof}
By Sion's minimax  theorem \cite{Sion},   \eqref{obb-c}  has a saddle point $(\x^*,\y^*)\in\X\times\Y$ such that
\begin{equation}\label{eq-17}
c(\x^*,\y)\leq c(\x^*,\y^*)\leq c(\x,\y^*),\quad\forall\, \x\in\X,\,\y\in\Y.
\end{equation}

By the continuity of functions $g_i$ and $h_j$, there exists a $\delta>0$ such that
\begin{eqnarray*}
&g_i(\x)>0,\quad\forall i\in\mathcal{A}^+(\x^*),\,\x\in\mathbf{B}(\x^*,\delta)\cap\X,\label{eq-e7}\\
&h_j(\y)>0,\quad\forall j\in\mathcal{B}^+(\y^*),\,\y\in\mathbf{B}(\y^*,\delta)\cap\Y,\label{eq-e8}
\end{eqnarray*}
which implies $\mathcal{A}^+(\x^*)\subseteq\mathcal{A}^+(\x)$ and $\mathcal{B}^+(\y^*)\subseteq\mathcal{B}^+(\y)$ in a neighborhood of $(\x^*,\y^*)$.
On the one hand,
$${c}(\x^*,\y^*)=f(\x^*,\y^*)-\lambda_1\sum_{i\in\mathcal{A}^+(\x^*)}1
+\lambda_2\sum_{j\in\mathcal{B}^+(\y^*)}1.$$
On the other hand, for $\x\in \mathbf{B}(\x^*,\delta)\cap\X$ and
$\y\in \mathbf{B}(\y^*,\delta)\cap\Y$, it has
$$
c(\x,\y^*)=f(\x,\y^*)-\lambda_1\sum_{i\in\mathcal{A}^+(\x)}1
+\lambda_2\sum_{j\in\mathcal{B}^+(\y^*)}1
\leq f(\x,\y^*)-\lambda_1\sum_{i\in\mathcal{A}^+(\x^*)}1
+\lambda_2\sum_{j\in\mathcal{B}^+(\y^*)}1,
$$
and
$$
c(\x^*,\y)=f(\x^*,\y)-\lambda_1\sum_{i\in\mathcal{A}^+(\x^*)}1
+\lambda_2\sum_{j\in\mathcal{B}^+(\y)}1
\geq f(\x^*,\y)-\lambda_1\sum_{i\in\mathcal{A}^+(\x^*)}1
+\lambda_2\sum_{j\in\mathcal{B}^+(\y^*)}1.
$$
Thus, we can conclude that
$$f(\x^*,\y)\leq f(\x^*,\y^*)\leq f(\x,\y^*),\quad \forall\, \x\in \mathbf{B}(\x^*,\delta)\cap\X,\,
\y\in \mathbf{B}(\y^*,\delta)\cap\Y,$$
which implies that $(\x^*,\y^*)$ is a local saddle point of \eqref{obb}.
\end{proof}

Since (\ref{obb-c}) has a saddle point, Proposition \ref{prop4} ensures the existence of a local saddle point of problem \eqref{obb}.
Moreover,  by \cite{Sion}, \eqref{obb} also has a local saddle point if $c$ is a quasi-convex-concave function.

The following example shows that the parameters $\lambda_1$ and $\lambda_2$ play an important role for the relations between saddle points, local saddle points, global minimax points and local minimax points of  min-max problem \eqref{obb}.
\begin{example}\label{example-s}
Consider the following min-max problem
\begin{equation}\label{example1}
\min_{\x\in\X}\max_{\y\in\Y}f(\x,\y):=(\x-1)(\y-1)+\lambda_1\|\x\|_0-\lambda_2\|\y\|_0,
\end{equation}
where $\X=\Y=[-2,2]$ and $\lambda_1,\lambda_2>0$. It is clear that
$c(\x,\y)=(\x-1)(\y-1)$ is convex-concave on $\X\times\Y$ and $(1,1)$ is the unique saddle point of $\min_{\x\in\X}\max_{\y\in\Y}c(\x,\y)$.

Case 1 (has no saddle point): Let $\lambda_1=3$ and $\lambda_2=1$.
By simple calculation, we find
$$p(\x)=\max_{\y\in\Y}f(\x,\y)=\left\{\begin{aligned}
&\max\{5-3\x,\x+1,4-\x\}&&\mbox{\emph{if} $\x\neq0$}\\
&2&&\mbox{\emph{if} $\x=0$}
\end{aligned}\right.$$
and
$$q(\y)=\min_{\x\in\X}f(\x,\y)=\left\{\begin{aligned}
&\min\{-3\y+5,-\y,\y+1\}&&\mbox{\emph{if} $\y\neq0$}\\
&1&&\mbox{\emph{if} $\y=0$}.
\end{aligned}\right.$$
Hence, we have $\min_{\x\in\X}\max_{\y\in\Y}f(\x,\y)=f(0,-2)=2$ and
$\max_{\y\in\Y}\min_{\x\in\X}f(0,0)=1$.
Thus, \eqref{gap} fails in this case. By {\cite[Theorem 1.4.1]{Pang-book2}}, we can conclude that there is no saddle point to problem \eqref{example1} with $\lambda_1=3$ and $\lambda_2=1$. On the other hand, it has four local saddle points:
$(0,0)$, $(1,1)$, $(0,-2)$ and $(2,0)$.

Case 2 (has a saddle point): Let $\lambda_1=\lambda_2=3$.
The similar calculation gives
$$p(\x)=\max_{\y\in\Y}f(\x,\y)=\left\{\begin{aligned}
&\max\{\x-1,4-\x,-3\x+3\}&&\mbox{\emph{if} $\x\neq0$}\\
&1&&\mbox{\emph{if} $\x=0$}
\end{aligned}\right.$$
and
$$q(\y)=\min_{\x\in\X}f(\x,\y)=\left\{\begin{aligned}
&\min\{-2-\y,3-3\y,\y-1\}&&\mbox{\emph{if} $\y\neq0$}\\
&1&&\mbox{\emph{if} $\y=0$}.
\end{aligned}\right.$$
Then, $\min_{\x\in\X}\max_{\y\in\Y}f(\x,\y)=\max_{\y\in\Y}\min_{\x\in\X}f(\x,\y)=1$.
By \cite[Theorem 1.4.1]{Pang-book2}, there exists a saddle point to problem \eqref{example1} with $\lambda_1=\lambda_2=3$, and $(0,0)$ is the unique saddle point.

Case 3 (all global minimax points are not local minimax points): Let $\lambda_1=\lambda_2=1$.
We have
$$p(\x)=\max_{\y\in\Y}f(\x,\y)=\left\{\begin{aligned}
&\max\{3-3\x,2-\x,\x-1\}&&\mbox{\emph{if} $\x\neq0$}\\
&2&&\mbox{\emph{if} $\x=0$}.
\end{aligned}\right.$$
Then, $\x^*=\arg\min_{\x\in\X}p(\x)=\{{3}/{2}\}$ and $\arg\max_{\y\in\Y}f(\x^*,\y)=\{0,2\}$. Thus, the set of global minimax points of \eqref{example1} with $\lambda_1=\lambda_2=1$  contains only two points
$({3}/{2},0)$ and $({3}/{2},2)$, and $f({3}/{2},0)=f({3}/{2},2)=1/2$. Moreover, around
$\x^*=3/2$, for any $0<\delta<1/2$,
$\max_{\y'\in\{\y\in\Y: |\y|\leq \delta\}}f(\x,\y')=f(\x,0)=2-\x$  and $\max_{\y'\in\{\y\in\Y:|\y-2|\leq\delta\}}f(\x,\y')=f(\x,2)=\x-1$, which means  that neither $({3}/{2},0)$ nor $({3}/{2},2)$ is a local minimax point of \eqref{example1} with $\lambda_1=\lambda_2=1$.
\end{example}

Although Proposition \ref{prop4} shows that problem \eqref{obb} always has local saddle points, not all of them are interesting due to the sparsity.
Inspired by this, we introduce a class of strong local saddle points of \eqref{obb}.
 \begin{definition}\label{lif-sta}
For a given $\nu> 0$, we call $({\x}^*,{\y}^*)\in\X\times\Y$ a \textbf{$\nu$-strong local saddle point} of problem \eqref{obb}, if it is a local saddle point of \eqref{obb} and satisfies the lower bound property as follows:
\begin{equation}\label{lower1}
g_i({\x}^*)\not\in(0,\nu),\,\,\forall i\in[\hat{n}]\quad \mbox{and}\quad
h_j({\y}^*)\not\in(0,\nu),\,\,\forall j\in[\hat{m}].
\end{equation}
\end{definition}

For any local saddle point $(\x^*,\y^*)$ of \eqref{obb}, there exists a $\nu>0$ such that
\eqref{lower1} holds,
where we can set $\nu=\min\{1, g_i(\x^*),h_j(\y^*):i\in\mathcal{A}^+(\x^*),j\in\mathcal{B}^+(\y^*)\}$. However, for a given $\nu>0$, not all local saddle points of \eqref{obb} satisfy \eqref{lower1} (see Example \ref{example2.1}). In particular,
if $({\x}^*,{\y}^*)$ is a $\nu$-strong local saddle point of \eqref{obb-4}, then $$\mbox{$|\x^*_i|\not\in(0,\nu)$ and $|\y_j^*|\not\in(0,\nu)$, $\forall i\in[n]$, $j\in[m]$,}$$
which not only helps us distinguish the zero and nonzero elements efficiently, but also provides a solution with certain stability \cite{BC-SINUM,Chen-Xu-Ye}. Hence the study of $\nu$-strong local saddle points of \eqref{obb} is interesting and important.

\begin{example}\label{example2.1}
Consider the following min-max problem
\begin{equation}\label{eq2.9}
\min_{\x\in\X}\max_{\y\in\Y}\,f(\x,\y):=|\x_1+\x_2-1|(\y+1)+\|\x\|_0-3\|\y\|_0
\end{equation}
with $\X=[-1,1]^2$ and $\Y=[-1,1]$.
By simple calculation, we can find that \eqref{eq2.9} has three saddle points, i.e. $(0,0,0)^\top$, $(1,0,0)^\top$ and $(0,1,0)^\top$
with $$\min_{\x\in\X}\max_{\y\in\Y}f(\x,\y)=\max_{\y\in\Y}\min_{\x\in\X}\,f(\x,\y)=1.$$

The local saddle point set of \eqref{eq2.9} is
$$\mathcal{SL}:=\{(\x_1,1-\x_1,\y)^\top:\x_1\in[-1,1],\y\in[-1,1]\}\cup\{(0,0,0)^\top, (0,0,1)^\top\},$$
while the $\nu$-strong local saddle point set of \eqref{eq2.9} with $0<\nu<1$ is
$$\mathcal{SL}\cap\{(\x_1,\x_2,\y)^\top:\, |\x_1|\not\in(0,\nu),\,|\x_2|\not\in(0,\nu)\,\mbox{and}\,\,\,|\y|\not\in(0,\nu)\}.$$
Notice that the $\nu$-strong local saddle point set of \eqref{eq2.9} is a proper subset of its local saddle point set, and contains all saddle points of \eqref{eq2.9}.
\end{example}

For a given $\delta>0$, by \cite[Theorem 1.4.1]{Pang-book2}, we know that $(\bar{\x},\bar{\y})$ is a saddle point of
problem (\ref{obb}) on $(\X\cap \mathbf{B}(\bar{\x},\delta) ) \times (\Y\cap \mathbf{B}(\bar{\y},\delta))$ if and only if
\begin{equation}\label{eq-e-3}
\max_{\y\in\Y\cap \mathbf{B}(\bar{\y},\delta)}\min_{\x\in\X\cap \mathbf{B}(\bar{\x},\delta)}f(\x,\y)=f(\bar{\x},\bar{\y})=
\min_{\x\in\X\cap \mathbf{B}(\bar{\x},\delta)}\max_{\y\in\Y\cap \mathbf{B}(\bar{\y},\delta)}f(\x,\y).
\end{equation}
Hence $(\bar{\x},\bar{\y})\in\X\times\Y$ is a local saddle point of (\ref{obb}) if and only if there is a $\delta>0$ such that \eqref{eq-e-3} holds.

Next, we provide the relation between $\nu$-strong local saddle points of \eqref{obb} and $\nu$-strong local  saddle points of \eqref{obb-c} restricted to a certain set.
\begin{theorem}\label{theorem2}
For $\nu>0$ and $(\bar{\x},\bar{\y})\in\X\times\Y$,
let $\hat{\X}=\{\x \in \X : g_i(\x)\leq 0,\,\forall i\not\in\mathcal{A}^+(\bar{\x})\}$ and $\hat{\Y}=\{\y\in \Y : h_j(\y)\leq 0,\,\forall j\not\in\mathcal{B}^+(\bar{\y})\}$. Then
the following statements are equivalent.
\begin{itemize}
\item [{\rm (i)}] $(\bar{\x},\bar{\y})$ is a $\nu$-strong local saddle point of \eqref{obb};
\item [{\rm (ii)}] $(\bar{\x},\bar{\y})$ is a local saddle point of \eqref{obb-c} on $(\X\cap\hat{\X})\times(\Y\cap\hat{\Y})$ and satisfies \eqref{lower1}.
\end{itemize}
\end{theorem}
\begin{proof}
(i)$\Rightarrow$(ii).
 Suppose $(\bar{\x},\bar{\y})$ is a $\nu$-strong local saddle point of \eqref{obb}, then there exists a $\delta>0$ such that
\begin{equation}\label{eq-e-4}
f(\bar{\x},\y)\leq f(\bar{\x},\bar{\y})\leq f({\x},\bar{\y}),\quad \forall\, \x\in\X\cap\mathbf{B}(\bar{\x},\delta),\,\y\in\Y\cap\mathbf{B}(\bar{\y},\delta).
\end{equation}

For any $\x\in\X\cap \hat{\X}$, it holds that $\|g(\x)_+\|_0\leq \|g(\bar{\x})_+\|_0$. Rearranging the second inequality in \eqref{eq-e-4} gives
$$c(\bar{\x},\bar{\y})+\lambda_1\|g(\bar{\x})_+\|_0\leq c({\x},\bar{\y})+\lambda_1\|g({\x})_+\|_0,\quad \forall\, \x\in\X\cap \mathbf{B}(\bar{\x},\delta).$$
Thus, $\bar{\x}$ is a local minimizer of $c(\cdot,\bar{\y})$ on {$\X\cap\hat{\X}$}. Following the same way, the first inequality in \eqref{eq-e-4} gives that $\bar{\y}$ is a local maximizer of $c(\bar{\x},\cdot)$ on $\Y\cap\hat{\Y}$. Thus, (ii) holds.

(ii)$\Rightarrow$(i). Since $\bar{\x}$ is a local minimizer of $c(\cdot,\bar{\y})$ on ${\X}\cap\hat{\X}$ and satisfies \eqref{lower1}, there exists a
$\delta_1>0$ such that
$$c(\bar{\x},\bar{\y})\leq c(\x,\bar{\y}),\quad \forall\,\x\in\X\cap\hat{\X}\cap \mathbf{B}(\bar{\x},\delta_1)$$
and
\begin{equation}\label{eq-e-2}
g_i(\bar{\x})\not\in(0,\nu),\,\, \forall i\in[\hat{n}].
\end{equation}

Based on \eqref{eq-e-2}, there exists a $\delta_2\in(0,\delta_1]$ such that
$g_i({\x})>0$, $\forall\,\x\in \mathbf{B}(\bar{\x},\delta_2),\,i\in\mathcal{A}^+(\bar{\x})$,
which implies
\begin{equation}\label{eq-e-5}
\|g(\bar{\x})_+\|_0\leq \|g({\x})_+\|_0,\quad \forall\, \x\in \mathbf{B}(\bar{\x},\delta_2).
\end{equation}
Then,
\begin{equation}\label{eq-e-6}
f(\bar{\x},\bar{\y})\leq f(\x,\bar{\y}),\quad \forall\,\x\in\X\cap\hat{\X}\cap \mathbf{B}(\bar{\x},\delta_2).
\end{equation}

Due to the continuity of $c(\cdot,\bar{\y})$, there is a $\delta_3\in(0,\delta_2]$ such that
\begin{equation}
c(\bar{\x},\bar{\y})\leq c(\x,\bar{\y})+\lambda_1, \quad\forall\, \x\in \mathbf{B}(\bar{\x},\delta_3).
\end{equation}
When $\x\in\X\cap\mathbf{B}(\bar{\x},\delta_3)$ but ${\x}\not\in\hat{\X}$, there exists an $\hat{i}\not\in\mathcal{A}^+(\bar{\x})$ such that $g_{\hat{i}}({\x})>0$, which together with \eqref{eq-e-5} further gives
\begin{equation}\label{eq-e-7}
\|g(\bar{\x})_+\|_0+1\leq \|g({{\x}})_+\|_0.
\end{equation}

Thanks to \eqref{eq-e-6}-\eqref{eq-e-7}, we have
\begin{equation}\label{eq-e-8}
f(\bar{\x},\bar{\y})\leq f(\x,\bar{\y}),\quad \forall\,\x\in\X\cap\mathbf{B}(\bar{\x},\delta_3).
\end{equation}

We can ensure $f(\bar{\x},{\y})\leq f(\bar{\x},\bar{\y})$, $\forall\y\in\Y\cap\mathbf{B}(\bar{\y},\delta_4)$ with $\delta_4>0$ in the same way. Thus,
$(\bar{\x},\bar{\y})$ is a $\nu$-strong local saddle point of \eqref{obb}.
\end{proof}

For problem \eqref{obb-4},
$\hat{\X}=\{\x\in\X: \x_i=0, \forall i\in\mathcal{A}^0(\bar{\x})\}$ and $\hat{\Y}=\{\y\in \Y: \y_j=0,\,\forall j\in\mathcal{B}^0(\bar{\y})\}$.
Thus, $(\bar{\x},\bar{\y})$ is a $\nu$-strong local saddle point of \eqref{obb-4} if and only if $(\bar{\x},\bar{\y})$ is a  $\nu$-strong saddle point of \eqref{obb-c} in the subspace corresponding to nonzero components of $(\bar{\x},\bar{\y})$.

\section{Continuous relaxations}\label{section3}
In this section, we propose a class of  continuous relaxations to the cardinality function in min-max problem \eqref{obb} based on convolution \cite{Chen-Mangasarian}, which include the capped-$\ell_1$ function \cite{Peleg2008}, SCAD function \cite{Fan1997}, MCP function \cite{Zhang2011} and hard thresholding penalty function \cite{Fan1997} as special cases. Then, we show the existence of the local saddle points of the continuous relaxations of \eqref{obb}.
\subsection{Density functions}

Let $\rho:\mathbb{R}\rightarrow\mathbb{R}_+$ be a piecewise continuous density function satisfying
\begin{equation}\label{eq-ass1}
\mbox{$\rho(s)=0$, $\forall s\not\in[0,\alpha]$}
\end{equation}
with a positive number $\alpha$,  which means that $\int_{0}^{\alpha}\rho(s)ds=1$.
Then, for any fixed $\mu>0$,
\begin{equation}\label{c-re}
\begin{aligned}
r(t,\mu):=&\int_{-\infty}^{+\infty}((t-\mu s)_+)^0\rho(s)ds\\
=&\int_{-\infty}^{\frac{t}{\mu}}\rho(s)ds=(t_+)^0
+\left\{\begin{array}{lc}
           0 & \, \mbox{if $t\leq0$} \\
          -\int_{\frac{t}{\mu}}^{+\infty}\rho(s)ds & \, \mbox{if $t>0$}
         \end{array}
\right.
\end{aligned}\end{equation}
is  well-defined,
and
\begin{equation}\label{eq-19}
\partial_tr(t,\mu)={\rm co }\left\{\lim\nolimits_{s\rightarrow {t}}\frac{\rho(s/\mu)}{\mu}\right\}.
\end{equation}

By \eqref{c-re}, for any $\mu>0$, we have
\begin{eqnarray}
r(t,\mu)=(t_+)^0,\quad \forall\, t\in(-\infty,0]\cup[\alpha\mu,+\infty), \label{eq-rho}\\
r(t,\mu)-(t_+)^0=-\int_{t/\mu}^{+\infty}\rho(s)ds\leq0, \quad \forall\, t>0,\label{eq-18-1}\\
\lim\nolimits_{\mu\downarrow0}r(t,\mu)
=(t_+)^0,\quad \forall \,t\in\mathbb{R},\label{eq-18}\\
\lim\nolimits_{a\rightarrow{t},\mu\downarrow0}r(a,\mu)
=(t_+)^0,\quad \forall\, {t}\neq0.\label{eq-18-2}
\end{eqnarray}
{For any $t\in\R$, we see from \eqref{eq-18-1} and \eqref{eq-18} that $r(t,\mu)$ approximates $(t_+)^0$ from below as $\mu$ tends to 0.}
In what follows, we give four examples of the function $r$ with $\rho$ satisfying \eqref{eq-ass1}.
\begin{example}\label{example4.1}
Choose a density function with $\alpha=1$ as
$$\rho(s)=\left\{\begin{aligned}&1&&\mbox{\emph{if} $0\leq s\leq1$}\\
&0&&\mbox{\emph{otherwise}}\end{aligned}\right.\quad\Rightarrow
\quad r(t,\mu)=\left\{
\begin{aligned}
&0&&\mbox{\emph{if} $t\leq0$}\\
&\frac{t}{\mu}&&\mbox{\emph{if} $0<t\leq\mu$}\\
&1&&\mbox{\emph{if} $t>\mu$.}
\end{aligned}\right.$$
Here, $r(\cdot,\mu)$ is the \textbf{capped-$\ell_1$ function} $\varphi_{\tt{cap}}$\footnote{capped-$\ell_1$ function: $\varphi_{\tt{cap}}(t)=\min\{1,|t|/\mu\}$.}.
\end{example}
\begin{example}\label{example4.2}
For any $\alpha>1$, choose a density function as
$$\rho(s)=\left\{\begin{aligned}&\frac{2}{\alpha+1}&&\mbox{\emph{if} $0\leq s\leq1$}\\
&\frac{2\alpha-2s}{(\alpha-1)(\alpha+1)}&&\mbox{\emph{if} $1<s\leq \alpha$}\\
&0&&\mbox{\emph{otherwise}}
\end{aligned}\right.
\quad\Rightarrow
\quad
r(t,\mu)=\left\{
\begin{aligned}
&0&&\mbox{\emph{if} $t\leq0$}\\
&\frac{2t}{(\alpha+1)\mu}&&\mbox{\emph{if} $0<t\leq\mu$}\\
&\frac{2\alpha\mu t-t^2-\mu^2}{(\alpha-1)(\alpha+1)\mu^2}&&\mbox{\emph{if} $\mu<t\leq \alpha\mu$}\\
&1&&\mbox{\emph{if} $t>\alpha\mu$}.
\end{aligned}\right.$$
Here, $r(\cdot,\mu)$  is a scaled \textbf{SCAD function} $\varphi_{\tt{SCAD}}$\footnote{SCAD function: $\varphi_{\tt{SCAD}}(t)=\left\{
\begin{aligned}
& t\quad &&\mbox{if $t\leq\mu$}\\
&\frac{2\alpha\mu t-t^2-\mu^2}{2(\alpha-1)\mu}\quad &&\mbox{if $\mu<t\leq \alpha\mu$}\\
&\frac{(\alpha+1)\mu}{2}\quad &&\mbox{if $t>\alpha\mu$.}
\end{aligned}\right.
$},
i.e. $r(t,\mu)=\frac{2}{(\alpha+1)\mu}\varphi_{\tt{SCAD}}(t)$, $\forall t
\in\R_+$.
\end{example}
\begin{example}\label{example4.3}
For any $\alpha>0$, choose a density function as
$$\rho(s)=\left\{\begin{aligned}&\frac{2}{\alpha}-\frac{2s}{\alpha^2}&&\mbox{\emph{if} $0\leq s\leq \alpha$}\\
&0&&\mbox{\emph{otherwise}}
\end{aligned}\right.
\quad\Rightarrow
\quad
r(t,\mu)=\left\{
\begin{aligned}
&0&&\mbox{\emph{if} $t\leq0$}\\
&\frac{2t}{\alpha\mu}-\frac{t^2}{\alpha^2\mu^2}&&\mbox{\emph{if} $0<t\leq \alpha\mu$}\\
&1&&\mbox{\emph{if} $t>\alpha\mu$}.
\end{aligned}\right.$$
Here, $r(\cdot,\mu)$ is a scaled \textbf{MCP function} $\varphi_{\tt{MCP}}$\footnote{MCP function: $\varphi_{\tt{MCP}}(t)=\left\{\begin{aligned}
&\frac{\alpha\mu}{2}\quad &&\mbox{if $t\geq
\alpha\mu$}\\
&t-\frac{t^2}{2\alpha\mu}&&\mbox{if $t< \alpha\mu$.}\end{aligned}\right.
$},
i.e. $r(t,\mu)=\frac{2}{\alpha\mu}\varphi_{\tt{MCP}}(t)$, $\forall t
\in\R_+$.
\end{example}
\begin{example}\label{example4.4}
Choose a density function with $\alpha=1$ as
$$\rho(s)=\left\{\begin{aligned}&2(1-s)&&\mbox{\emph{if} $0<s<1$}\\
&0&&\mbox{\emph{otherwise}}
\end{aligned}\right.
\quad\Rightarrow
\quad
r(t,\mu)=\left\{
\begin{aligned}
&0&&\mbox{\emph{if} $t\leq0$}\\
&1-(1-{t}/{\mu})^2&&\mbox{\emph{if} $0<t<\mu$}\\
&1&&\mbox{\emph{if} $t>\mu$.}
\end{aligned}\right.$$
Here, $r(\cdot,\mu)$ is the \textbf{hard thresholding penalty function} $\varphi_{\tt{hard}}$\footnote{hard thresholding penalty function: $\varphi_{\tt{hard}}(t)=1-(1-t/\mu)_+^2
$.}.
\end{example}

For further analysis, we bring forward  two assumptions on density function $\rho$.
\begin{assumption}\label{ass7}
There exist positive numbers $\underline{\rho}$ and $\bar{\rho}$ such that the density function $\rho:\mathbb{R}\rightarrow\mathbb{R}_+$ satisfies
$$\underline{\rho}\leq \rho(s)\leq \overline{\rho},\quad \forall s\in(0,\alpha).$$
\end{assumption}

\begin{assumption}\label{ass8}
The density function $\rho$ is Lipschitz continuous on $\mathbb{R}_{++}$ and
there exist $\underline{\rho}_2>0$ and $\check{\rho}_2>0$ such that for any $s\in(0,\alpha)$,
$$
\mbox{either} \,\; \rho(s)\geq \underline{\rho}_2\,\;\; \mbox{or}\,\;\;
{\sup\{a:{a\in\partial\rho(s)}\}\leq-\check{\rho}_2}.$$
\end{assumption}

Notice that if $\rho$ is Lipschitz continuous on $\mathbb{R}_{++}$ and  $\rho(s)=0$, $\forall s\not\in[0,\alpha]$, then Assumption \ref{ass7} fails. Thus Assumption \ref{ass7} and Assumption \ref{ass8} are two different situations.

When the density function $\rho$ satisfies Assumption \ref{ass7}, we have
$$\inf\{\xi:\xi\in\partial _tr(t,\mu)\}\geq\underline{\rho}/\mu, \quad \forall\, t\in(0,\alpha\mu).$$

When the density function $\rho$ satisfies Assumption \ref{ass8},
inspired by \eqref{eq-19}, we have that  for any $\mu>0$, $r(\cdot,\mu)$ is Lipschitz continuously differentiable on $\R_{++}$ and satisfies
\begin{equation}\label{sec-r}
\partial^2_tr(t,\mu)=\partial\rho(t/\mu)/\mu^2, \quad \forall\, t>0,
\end{equation}
which implies for any $t\in(0,\alpha\mu)$ such that $\sup\{a:a\in\partial\rho(t/\mu)\}\leq-\check{\rho}_2$, it holds
$$\sup\{\xi:\xi\in\partial^2_tr(t,\mu)\}\leq- \check{\rho}_2/\mu^2.$$

{Since all the four density functions $\rho$ in Examples \ref{example4.1}-\ref{example4.4} satisfy \eqref{eq-ass1}, $r(t,\mu)$ in these examples satisfy \eqref{eq-rho}-\eqref{eq-18}. To end this subsection, we use Table \ref{table1} to conclude the different properties of the density functions and the corresponding continuous functions $r$ in these four examples.}
\begin{table}[htbp]
\centering{
\begin{tabular}{|c|c|c|c|} \hline
Example &differentiability of $r(\cdot,\mu)$ &Assumption \ref{ass7} &Assumption \ref{ass8}\\
\hline
\ref{example4.1}& not at $0$, $\mu$& $\underline{\rho}=\overline{\rho}=1$&$\times$\\
\hline
\ref{example4.2}& not at $0$&$\times$& $\underline{\rho}_2=\frac{2}{\alpha+1}$, $\check{\rho}_2=\frac{2}{(\alpha+1)(\alpha-1)}$\\
\hline
\ref{example4.3}& not at $0$&$\times$& $\underline{\rho}_2>0$, $\check{\rho}_2=\frac{2}{\alpha^2}$\\
\hline
\ref{example4.4}& not at $0$&$\times$& $\underline{\rho}_2>0$, $\check{\rho}_2=2$\\
\hline
\end{tabular}}\caption{Properties of the density functions $\rho$ and corresponding functions $r$ in Examples \ref{example4.1}-\ref{example4.4}}\label{table1}
\end{table}
\subsection{Continuous relaxation models to \eqref{obb}}
In what follows, we will use the continuous function $r$ defined in \eqref{c-re} to
approximate the cardinality function in \eqref{obb}.
For $i\in[\hat{n}]$ and $j\in[\hat{m}]$, denote
$$\phi_{i}(\x)=(g_i(\x)_+)^0\quad\mbox{and}\quad
\psi_{j}(\y)=(h_j(\y)_+)^0,$$
and define their continuous relaxations by
\begin{equation}\label{eq-con1}
{\phi}_{i}^{R}(\x,\mu)=r(g_i(\x),\mu) \quad
\mbox{and}\quad{\psi}_{j}^{R}(\y,\mu)=r(h_j(\y),\mu).
\end{equation}
For any $\mu>0$, $i\in[\hat{n}]$ and $j\in[\hat{m}]$, by \eqref{eq-18-1}, we have
\begin{equation}\label{eq-re2}
 \phi_{i}^R(\x,\mu)\leq \phi_{i}(\x)\quad
\mbox{and}\quad\psi_{j}^R(\y,\mu)\leq \psi_{j}(\y), \quad \quad \forall \, \x\in\mathcal{X}, \,\y\in\mathcal{Y}.
\end{equation}

We propose the continuous relaxation of \eqref{obb} as follows
\begin{equation}\label{ob-r}
\min_{\x\in\X}\max_{\y\in\Y} f^R(\x,\y,\mu):=c(\x,\y)+\lambda_1\sum_{i=1}^{\hat{n}}\phi_{i}^R(\x,\mu)-\lambda_2
\sum_{j=1}^{\hat{m}}\psi_{j}^R(\y,\mu),
\end{equation}
where $\mu$ is a given positive number.
Here, $f^R(\cdot,\cdot,\mu)$ in \eqref{ob-r} is Lipschitz continuous on $\X\times\Y$, and it is clear by \eqref{eq-18} that,
$\lim_{\mu\downarrow0}f^R(\x,\y,\mu)=f(\x,\y)$ for any $\x\in\mathcal{X}$ and $\y\in\mathcal{Y}$.

Notice that \eqref{ob-r} is a nonconvex-nonconcave min-max problem and may not have a
(local) saddle point.
However, similar to Proposition \ref{prop4},
we can have the existence results for the local saddle points of \eqref{ob-r}.
\begin{proposition}\label{prop5}
There exists a $\tilde{\mu}>0$ such that \eqref{ob-r} has a local saddle point for any $\mu\in(0,\tilde{\mu})$.
\end{proposition}
\begin{proof}
Let $(\x^*,\y^*)$ be a saddle point of $\min_{\x\in\X}\max_{\y\in\Y}c(\x,\y)$, i.e. \eqref{eq-17} holds.
Denote
$\vartheta=\min\{1, g_i(\x^*),\,h_j(\y^*):i\in\mathcal{A}^+(\x^*),j\in\mathcal{B}^+(\y^*)\}$
 and set
$\tilde{\mu}=\vartheta/2\alpha$, then $g_i(\x^*)\geq2\alpha\tilde{\mu}$, $h_j(\y^*)\geq2\alpha\tilde{\mu}$, $\forall i\in\mathcal{A}^+(\x^*)$, $j\in\mathcal{B}^+(\y^*)$.  Choose $\mu\in(0,\tilde{\mu})$. By the continuity of $g$ and $h$, there exists a $\delta_1>0$ such that for any
$i\in\mathcal{A}^+(\x^*)$, $j\in\mathcal{B}^+(\y^*)$, $\x\in \textbf{B}(\x^*,\delta)$ and $\y\in\textbf{B}(\y^*,\delta)$,  it holds
$g_i(\x)\geq\alpha\mu$ and $h_j(\y)\geq\alpha\mu$,
by \eqref{eq-rho}, which further implies
$$
\phi_{i}^R(\x,\mu)=1\quad\mbox{and}\quad \psi_j^R(\y,\mu)=1.$$
This means that, for any $\x\in \textbf{B}(\x^*,\delta)$ and $\y\in\textbf{B}(\y^*,\delta)$,
\begin{equation}\label{eq2.3}
\sum_{i=1}^{\hat{n}}\phi_{i}^R(\x^*,\mu)\leq\sum_{i=1}^{\hat{n}}\phi_{i}^R(\x,\mu)\quad\mbox{and}\quad
\sum_{j=1}^{\hat{m}}\psi_{j}^R(\y^*,\mu)\leq\sum_{j=1}^{\hat{m}}\psi_{j}^R(\y,\mu).
\end{equation}
Together \eqref{eq-17} with \eqref{eq2.3}, we obtain
$$f^R(\x^*,\y,\mu)\leq f^R(\x^*,\y^*,\mu)\leq f^R(\x,\y^*,\mu),\quad \x\in \mathbf{B}(\x^*,\delta)\cap\X,\,
\y\in \mathbf{B}(\y^*,\delta)\cap\Y,$$
which means that $(\x^*,\y^*)$ a local saddle point of \eqref{ob-r}.
\end{proof}

From the compactness of $\X$ and $\Y$, \eqref{ob-r} has global minimax points for any $\mu>0$.
\section{Theoretical analysis on exact continuous relaxations}\label{section4}
In this section, we will consider the consistence of problem \eqref{obb} and its continuous relaxation problem \eqref{ob-r} with the  density function $\rho$ satisfying
\eqref{eq-ass1}. Moreover, the smoothing approximation to a nonsmooth function $c$ is defined and discussed in subsection \ref{sec-smoothing}.
\subsection{Relations on saddle points}\label{sec4.1}
To proceed the discussion on the saddle points and local saddle points between problem \eqref{obb} and its
 continuous relaxation model \eqref{ob-r}, we need the following assumption, which will be discussed and verified in detail in Section \ref{section5}.
\begin{assumption}\label{ass6} For a given $\mu>0$,
the following conditions hold.
\begin{itemize}
\item [{\rm (i)}] For any $\bar{\y}\in\Y$, if $\x^*$ is a local minimizer of $f^R(\x,\bar{\y},\mu)$ on $\X$, then
\begin{equation}\label{eq-ass6-1}
g_i({\x}^*)\not\in(0,\alpha\mu), \quad\forall i\in[\hat{n}].
\end{equation}
\item [{\rm (ii)}] For any $\bar{\x}\in\X$, if $\y^*$ is a local maximizer of $f^R(\bar{\x},\y,\mu)$ on $\Y$, then
\begin{equation}\label{eq-ass6-2}
h_j({\y}^*)\not\in(0,\alpha\mu), \quad \forall j\in[\hat{m}].
\end{equation}
\end{itemize}
\end{assumption}

Under Assumption \ref{ass6}, if $(\x^*,\y^*)$ satisfies the lower bounds in \eqref{eq-ass6-1} and \eqref{eq-ass6-2}, by \eqref{eq-rho}, then
\begin{equation}\label{rela2}
\phi_{i}^R(\x^*,\mu)=\phi_{i}(\x^*),\quad\forall i\in[\hat{n}]\quad\mbox{and}\quad \psi_{j}^R(\y^*,\mu)=\psi_{j}(\y^*),\,\forall j\in[\hat{m}].
\end{equation}
Moreover, if Assumption \ref{ass6} holds for $\hat{\mu}$, then it also holds for any $\mu\in(0,\hat{\mu}]$.

Inspired by \eqref{rela2}, we have the following relations on the saddle points and local saddle points between problems \eqref{obb} and \eqref{ob-r}.
\begin{theorem}\label{theorem1}
Suppose problem \eqref{ob-r} satisfies Assumption \ref{ass6}, then
\begin{itemize}
\item [{\rm (i)}] $(\x^*,\y^*)$ is a saddle point of \eqref{obb} if and only if it is a
saddle point of \eqref{ob-r};
\item [{\rm (ii)}] $(\x^*,\y^*)$ is a local saddle point of \eqref{obb}, if it is a
local saddle point of \eqref{ob-r}.
\end{itemize}
\end{theorem}
\begin{proof}
Suppose $(\x^*,\y^*)$ is a global (local) saddle point of problem \eqref{ob-r}. By the relations in \eqref{eq-re2} and \eqref{rela2}, it holds
$f^R(\x^*,\y^*,\mu)=f(\x^*,\y^*)$,
$f^R(\x^*,\y,\mu)\geq f(\x^*,\y)$ and
$f^R(\x,\y^*,\mu)\leq f(\x,\y^*)$. Then, $(\x^*,\y^*)$ is a global (local) saddle point of problem \eqref{obb}. Thus we only need to prove that if $(\x^*,\y^*)$ is a saddle point of problem \eqref{obb}, then  it is a
saddle point of \eqref{ob-r}.

Assume on contradiction that $(\x^*,\y^*)$ is a saddle point of problem \eqref{obb}, but it is not a saddle point  of \eqref{ob-r}. Then
\begin{equation}\label{th1-eq1}
\x^*\in\arg\min_{\x\in\X}f(\x,\y^*)\quad\mbox{and}\quad \y^*\in\arg\max_{\y\in\Y}f(\x^*,\y),
\end{equation}
but
either $\x^*$ is not a global minimizer of
$f^R(\x,\y^*,\mu)$ on $\X$ or $\y^*$ is not a global maximizer of
$f^R(\x^*,\y,\mu)$ on $\Y$. As a possible situation,
if $\x^*$ is not a global minimizer of
$f^R(\x,\y^*,\mu)$ on $\X$, then there exists $\bar{\x}\in\arg\min_{\x\in\X}f^R(\x,\y^*,\mu)$ such that
\begin{equation}\label{eq2.1}
f^R(\bar{\x},\y^*,\mu)<f^R(\x^*,\y^*,\mu).
\end{equation}
By \eqref{eq-ass6-1} in Assumption \ref{ass6}, either $g_i(\bar{\x})\geq\alpha\mu$ or $g_i(\bar{\x})\leq0$, which means $\phi_{i}^R(\bar{\x},\mu)=\phi_{i}(\bar{\x})$,
$\forall i\in[\hat{n}]$, and then
\begin{equation}\label{eq2.1-2}
f^R(\bar{\x},\y^*,\mu)=f(\bar{\x},\y^*)-\lambda_2\sum_{j\in[\hat{m}]}\psi_j^R(\y^*,\mu)
+\lambda_2\sum_{j\in[\hat{m}]}\psi_j(\y^*).
\end{equation}
While by $\phi_i^R({\x}^*,\mu)\leq \phi_i({\x}^*)$, $\forall i\in[\hat{n}]$, we obtain
\begin{equation}\label{eq2.1-3}
f^R({\x}^*,\y^*,\mu)\leq f({\x}^*,\y^*)-\lambda_2\sum_{j\in[\hat{m}]}\psi_j^R(\y^*,\mu)+\lambda_2\sum_{j\in[\hat{m}]}\psi_j(\y^*).
\end{equation}

Combining \eqref{eq2.1}-\eqref{eq2.1-3}, we find that
$f(\bar{\x},\y^*)<f(\x^*,\y^*)$, which contradicts to the first relation in \eqref{th1-eq1}. Thus, $\x^*$ is a global minimizer of $f^R(\x,\y^*,\mu)$ on $\X$. And we can verify that $\y^*$ is a global maximizer of $f^R(\x^*,\y,\mu)$ on $\Y$ by a similar way. Therefore, $(\x^*,\y^*)$ is a saddle point of problem \eqref{ob-r}.
\end{proof}
\begin{remark}
Following the proof of Theorem \ref{theorem1}, we see that $f^R(\x^*,\y,\mu)\geq f(\x^*,\y)$ for any $\y\in\Y$ and
$f^R(\x,\y^*,\mu)\leq f(\x,\y^*)$ for any $\x\in\X$ are the key points to guarantee that any (local) saddle point of \eqref{ob-r} is that of \eqref{obb}, while the lower bounds in Assumption \ref{ass6} are the key properties to give the inverse relation on their saddle points. Moreover, by Theorem \ref{theorem1}, we confirm that any saddle point $(\x^*,\y^*)$ of \eqref{obb} satisfies the lower bounds in \eqref{eq-ass6-1} and \eqref{eq-ass6-2} under Assumption \ref{ass6}.
\end{remark}
\subsection{Relations on minimax points}\label{sec4.2}
To establish the equivalent relation on global minimax points between problem \eqref{obb} and problem
 \eqref{ob-r}, we need the following assumption, which will be discussed and verified in Section \ref{section5}.
\begin{assumption}\label{ass2}
For a given $\mu>0$, the following conditions hold.
\begin{itemize}
\item [{\rm (i)}] If $\x^*$ is a global minimizer of $\max_{\y\in\Y}f^R(\x,{\y},\mu)$ on $\X$, then \eqref{eq-ass6-1} holds.
\item [{\rm (ii)}] For any $\bar{\x}\in\X$, if $\y^*$ is a global maximizer of $f^R(\bar{\x},\y,\mu)$ on $\Y$, then \eqref{eq-ass6-2} holds.
\end{itemize}
\end{assumption}

Assumption \ref{ass2} implies that any global minimax point $({\x}^*,{\y}^*)$ of \eqref{ob-r} satisfies the lower bounds in \eqref{eq-ass6-1} and \eqref{eq-ass6-2}.
Then, we can establish the following relations on the global minimax points between problems \eqref{obb} and \eqref{ob-r}.
\begin{theorem}\label{theorem3}
Under Assumption \ref{ass2},
$(\x^*,\y^*)$ is a global minimax point of problem \eqref{ob-r} if and only if
it is a global minimax point of problem \eqref{obb}.
\end{theorem}
\begin{proof}
Let $({\x}^*,{\y}^*)$ be a global minimax point of problem \eqref{ob-r}, i.e.
\begin{equation}\label{eq-e15-o}
f^R(\x^*,\y,\mu)\leq f^R(\x^*,\y^*,\mu)\leq \max_{\y'\in\Y}f^R(\x,\y',\mu),\quad \forall\, \x\in\X,\,\y\in\Y,
\end{equation}
which implies $f^R(\x^*,\y^*,\mu)=\max_{\y'\in\Y}f^R(\x^*,\y',\mu)\leq \max_{\y'\in\Y}f^R(\x,\y',\mu)$.
Then, Assumption \ref{ass2} gives $\phi_{i}(\x^*)=\phi_{i}^R(\x^*,\mu)$ and $\psi_{j}(\y^*)=\psi_{j}^R(\y^*,\mu)$, and we further have
\begin{equation}\label{eq2.11}
f^R(\x^*,\y^*,\mu)=f(\x^*,\y^*).
\end{equation}
Invoking \eqref{eq-re2}, we have
\begin{equation}\label{eq2.7}
f(\x^*,\y)=c(\x^*,\y)+\lambda_1\sum_{i\in[\hat{n}]}\phi_{i}(\x^*)-\lambda_2\sum_{j\in[\hat{m}]}\psi_{j}(\y)
\leq f^R(\x^*,\y,\mu), \quad \forall \,  \y\in\Y.
\end{equation}
For any $\x\in\X$, denote $\y_{\x}$ a maximizer of $f^R(\x,\y,\mu)$ on $\Y$. Recalling Assumption \ref{ass2}, we have that $\y_{\x}$ satisfies the lower bound property in \eqref{eq-ass6-2} and $\psi_j(\y_{\x})=\psi_j^R(\y_{\x},\mu)$. Together it with \eqref{eq-re2}, we have that
\begin{equation}\label{eq2.8}
\begin{aligned}
\max_{\y'\in\Y}f^R(\x,\y',\mu)=&c(\x,\y_{\x})+\lambda_1\sum_{i\in[\hat{n}]}
\phi_i^R(\x,\mu)-\lambda_2\sum_{j\in[\hat{m}]}\psi_j^R(\y_{\x},\mu)\\
\leq&c(\x,\y_{\x})+\lambda_1\sum_{i\in[\hat{n}]}\phi_i(\x)-\lambda_2\sum_{j\in[\hat{m}]}\psi_j(\y_{\x})\leq\max_{\y'\in\Y}f(\x,\y').
\end{aligned}\end{equation}
Combining \eqref{eq-e15-o}-\eqref{eq2.8}, we confirm that $(\x^*,\y^*)$ is a global minimax point of problem \eqref{obb}.

Conversely, let $(\x^*,\y^*)$ be a global minimax point of problem \eqref{obb}, i.e.
\begin{equation}\label{eq-e15}
f(\x^*,\y)\leq f(\x^*,\y^*)\leq \max_{\y'\in\Y}f(\x,\y'), \quad \forall\, \x\in\X,\,\y\in\Y.
\end{equation}
By the first inequality in \eqref{eq-e15}, similar to the proof of Theorem \ref{theorem1}, we have
\begin{equation}\label{eq2.14}
f^R(\x^*,\y,\mu)\leq f^R(\x^*,\y^*,\mu),\quad\forall\,  \y\in\Y.
\end{equation}

Next, recalling Assumption \ref{ass2}, ${\y}^*\in\arg\max_{\y'\in\Y}f^R(\x^*,\y',\mu)$ implies
\begin{equation}\label{eq-16}
\psi_j(\y^*)=\psi_j^R(\y^*,\mu), \quad\forall\, j\in[\hat{m}],
\end{equation}
which together with $\phi_i^R(\x^*,\mu)\leq \phi_i(\x^*)$, $\forall i\in[\hat{n}]$ gives
\begin{equation}\label{eq2.20}
f^R(\x^*,\y^*,\mu)\leq f(\x^*,\y^*).
\end{equation}

Denote $(\bar{\x},\bar{\y})$ a global minimax point of \eqref{ob-r}, then
\begin{equation}\label{eq2.16}
f^R(\bar{\x},\y,\mu)\leq f^R(\bar{\x},\bar{\y},\mu)\leq\max_{\y'\in\Y} f^R(\x,\y',\mu),\quad \forall\, \x\in\X,\,\y\in\Y,
\end{equation}
by the first part of this theorem, which implies
\begin{equation}\label{eq2.18}
f(\bar{\x},\y)\leq f(\bar{\x},\bar{\y})\leq\max_{\y'\in\Y} f(\x,\y'),\quad \forall\, \x\in\X,\,\y\in\Y.
\end{equation}

By Assumption \ref{ass2}, we further have
\begin{equation}\label{eq-e13}
\phi_i^R(\bar{\x},\mu)=\phi_i(\bar{\x}),\,\forall i\in[\hat{n}],\,\,
\psi_j^R(\bar{\y},\mu)=\psi_j(\bar{\y}),\,\forall j\in[\hat{m}]\,\,\mbox{and}\,\,
f^R(\bar{\x},\bar{\y},\mu)=f(\bar{\x},\bar{\y}).
\end{equation}

In the second inequality of \eqref{eq-e15}, let $\x=\bar{\x}$ and use \eqref{eq2.18}, then we obtain
$f(\x^*,\y^*)\leq f(\bar{\x},\bar{\y})$,
which together with \eqref{eq2.20} and the third inequality in \eqref{eq-e13} gives
\begin{equation}\label{eq-e14}
f^R(\x^*,\y^*,\mu)\leq f^R(\bar{\x},\bar{\y},\mu).
\end{equation}

Combining \eqref{eq-e14} and
the second inequality in \eqref{eq2.16} gives
$$f^R(\x^*,\y^*,{\mu})\leq\max_{\y'\in\Y}f^R(\x,\y',\mu), \quad\forall\, \x\in\X,$$
which together with \eqref{eq2.14} guarantees that $(\x^*,\y^*)$ is a global minimax point of problem \eqref{ob-r}.
\end{proof}
\subsection{{Smoothing functions to a nonsmooth convex-concave function $c$}}\label{sec-smoothing}
If the function $c$ in problem \eqref{ob-r} is nonsmooth, the smoothing approximation of it is often needed in the algorithms \cite{chen-MP,Burke}. In what follows, we introduce a smoothing function of $c$ defined in \cite{chen-MP}.
\begin{definition}\label{defn2}
We call $\tilde{c}:\X\times\Y\times(0,1]\rightarrow\mathbb{R}$ a smoothing function of a nonsmooth function $c$ on $\X\times\Y$, if $\tilde{c}(\cdot,\cdot,\varepsilon)$ is continuously differentiable on
$\X\times\Y$ for any fixed $\varepsilon\in(0,1]$ and for any $(\bar{\x},\bar{\y})\in\X\times\Y$, it satisfies
\begin{equation}\label{eq-defn-s}
\lim\nolimits_{\x\to \bar{\x}, \y\to \bar{\y}, \varepsilon\downarrow0}\tilde{c}(\x,{\y},\varepsilon)=c(\bar{\x},\bar{\y}).
\end{equation}
\end{definition}

The gradient consistence between the Clarke subgradient of the nonsmooth function and the gradients associated with its smoothing function sequence is important for the efficiency of the smoothing method, i.e. for any $\bar{\x}\in\X$ and $\bar{\y}\in\Y$,
\begin{equation}\label{eqs-1}
\{\lim\nolimits_{\x\to \bar{\x}, \y\to \bar{\y}, \varepsilon\downarrow0}\nabla_{(\x,\y)}\tilde{c}({\x},{\y},\varepsilon)\}\subseteq\partial c(\bar{\x},\bar{\y}).
\end{equation}
The partial gradient consistences with respect to the update of two variables are often necessary
for the algorithm analysis of the min-max problems, i.e. for any $\bar{\x}\in\X$ and $\bar{\y}\in\Y$,
\begin{eqnarray}
&\left\{\lim_{\x\rightarrow \bar{\x},\y\rightarrow \bar{\y},\varepsilon\downarrow0}\nabla_{\x}\tilde{c}(\x,\y,\varepsilon)\right\}\subseteq\partial_{\x}c(\bar{\x},\bar{\y}),\label{s2}\\
&\left\{\lim_{\x\rightarrow\bar{\x},\y\rightarrow \bar{\y},\varepsilon\downarrow0}\nabla_{\y}\tilde{c}(\x,\y,\varepsilon)\right\}\subseteq\partial_{\y}c(\bar{\x},\bar{\y}).\label{s3}
\end{eqnarray}

However,
neither $\partial c(\x,\y)$ nor $\partial_{
\x} c(\x,\y)\times\partial_{\y} c(\x,\y)$ are contained in each other generally \cite{Clarke}.
See \cite[Example 2.5.2]{Clarke}. When $c$ is Clarke regular with respect to $(\x,\y)$, by \cite[Proposition 2.3.15]{Clarke}, it holds
\begin{equation}\label{s1}
\partial c(\x,\y)\subseteq\partial_{\x} c(\x,\y)\times\partial_{\y} c(\x,\y).
\end{equation}
In what follows, we will show that \eqref{s1} holds
for any convex-concave function $c$, though the convexity-concavity of $c$ cannot give the regularity of it. For example, $c(\x,\y)=|\x|-|\y|$ is convex-concave on $\R\times\R$, but not Clarke regular in $(\x,\y)$ at $(1,0)$.
\begin{proposition} \label{prop8} For any convex-concave function $c$, \eqref{s1} holds for any $\x\in\R^n$ and $\y\in\R^m$.
\end{proposition}
\begin{proof}
Let $(\xi,\eta)\in\partial c(\x,\y)$. We will prove that $\xi\in \partial_{\x} c(\x,\y)$, $\eta\in \partial_{\y} c(\x,\y)$.

Since $c(\cdot,\y)$ is convex on $\R^n$ for any $\y$,
by \cite[Proposition 2.5.3]{Clarke}, it has
$\xi\in \partial_{\x} c(\x,\y)$.
Inspired by the result in \cite[Proposition 2.3.1]{Clarke}, $(-\xi,-\eta)\in\partial(-c(\x,\y))$.
Using the concavity of $c(\x,\cdot)$ on $\R^m$ for any $\x\in\R^n$, $-c(\x,\y)$ is convex with respect to $\y$ and then $-\eta\in\partial_{\y}(-c(\x,\y))=-\partial_{\y}c(\x,\y)$, which uses \cite[Proposition 2.5.3]{Clarke} again. Thus, $\eta\in\partial_{\y}c(\x,\y)$.
\end{proof}

For a nonsmooth function $c$, we can construct a smoothing function of $c$ by convolution \cite{chen-MP,Qi-Chen,Rock-Wets} as follows
\begin{equation}\label{smoothing}
\tilde{c}(\z,\varepsilon)=\int_{\R^{n+m}}c(\z-\u)\psi_{\varepsilon}(\u)d\u,
\end{equation}
where $\z:=(\x,\y)$,  $\psi_{\varepsilon}:\R^{n+m}\rightarrow\R_+$ is a sequence of bounded, measurable functions satisfying {$\int_{\R^{n+m}}\psi_{\varepsilon}(\u)d\u=1$} and $\lim_{\varepsilon\downarrow0}B^{\varepsilon}=\{\bm{0}\}$ with $B^{\varepsilon}:=\{\u:\psi_{\varepsilon}(\u)>0\}$.
\begin{proposition}\label{prop-c}
Let $\tilde{c}:\R^n\times\R^m\times(0,1]\rightarrow\mathbb{R}$ be defined as in \eqref{smoothing}. Then  $\tilde{c}$ is
a smoothing function of $c$ on $\X\times\Y$ and satisfies the following properties:
\begin{itemize}
\item [{\rm (i)}] for any $\bar{\x}\in\X$ and $\bar{\y}\in\Y$, \eqref{eqs-1} and \eqref{s2}-\eqref{s3} hold;
    \item [{\rm (ii)}] for any $\varepsilon>0$, $\tilde{c}(\x,\y,\varepsilon)$ is convex in $\x\in\R^n$ and concave in $\y\in\R^m$.
\end{itemize}
\end{proposition}
\begin{proof}
From  \cite[Theorem 9.67]{Rock-Wets}, $\tilde{c}$ in \eqref{smoothing} is a smoothing function of $c$ on $\X\times\Y$ and satisfies $\{\lim\nolimits_{\x\to \bar{\x}, \y\to \bar{\y}, \varepsilon\downarrow0}\nabla\tilde{c}({\x},{\y},\varepsilon)\}\subseteq\partial c(\bar{\x},\bar{\y})$ for any $\bar{\x}\in\X$ and $\bar{\y}\in\Y$.
Recalling the convexity-concavity of $c$, by Proposition \ref{prop8}, \eqref{s2}-\eqref{s3} hold.

In what follows, we first verify that $\tilde{c}(\x,\y,\varepsilon)$ is convex in $\x\in\R^n$ for any $\y\in\R^m$ and $\varepsilon>0$.  For any $\bar{\x},\hat{\x}\in\X$ and $\eta\in[0,1]$, observe that
$$\begin{aligned}
&\tilde{c}(\eta\bar{\x}+(1-\eta)\hat{\x},\y,\varepsilon)=\int_{\R^{n+m}}
c(\eta\bar{\x}+(1-\eta)\hat{\x}-\v,\y-\w)\psi_{\varepsilon}(\v,\w)d\u\\
\leq&\eta\int_{\R^{n+m}}
c(\bar{\x}-\v,\y-\w)\psi_{\varepsilon}(\u)d\u+
(1-\eta)\int_{\R^{n+m}}
c(\hat{\x}-\v,\y-\w)\psi_{\varepsilon}(\u)d\u\\
=&\eta\tilde{c}(\bar{\x},\y,\varepsilon)+(1-\eta)\tilde{c}(\hat{\x},\y,\varepsilon),
\end{aligned}$$
where $\u=(\v,\w)\in \R^{n+m}$, the inequality uses the convexity of $c(\cdot,\y)$ and the nonnegativity of $\psi_{\varepsilon}$ on $\R^{n+m}$. Thus,
$\tilde{c}(\x,\y,\varepsilon)$ is convex in $\x\in\R^n$ for any $\y\in\R^m$ and $\varepsilon>0$. By similar calculation, $\tilde{c}(\x,\y,\varepsilon)$ is concave in $\y\in\R^m$ for any $\x\in\R^n$ and $\varepsilon>0$.
\end{proof}

Then, denote the smoothing model of \eqref{ob-r} by
\begin{equation}\label{ob-cr}
\min_{\x\in\X}\max_{\y\in\Y} \tilde{f}^R(\x,\y,\mu,\varepsilon):=\tilde{c}(\x,\y,\varepsilon)+\lambda_1\sum_{i=1}^{\hat{n}} {\phi}_{i}^{R}(\x_i,\mu)
-\lambda_2\sum_{j=1}^{\hat{m}}{\psi}_{j}^{R}(\y_j,\mu),
\end{equation}
where $\tilde{c}(\x,\y,\varepsilon)$ is a smoothing function of $c$. Since $\tilde{c}$ is a convex-concave function, it always has a saddle point over $\X\times \Y$. Following the results in previous sections, problem (\ref{ob-cr}) has a local saddle point and
a global minimax point for some $\mu>0$.

\section{A particular case  of problem \eqref{obb}}\label{section5}
In Section 4, we show the consistency on the saddle point sets and the inclusion on the local saddle point sets of problems \eqref{obb} and \eqref{ob-r} under Assumption \ref{ass6}, and the consistence on their global minimax point sets under Assumption \ref{ass2}. In this section, we will verify that
Assumptions \ref{ass6} and \ref{ass2} hold for a particular case of problem \eqref{obb}. Moreover, with the specific structure of $r$ under Assumption \ref{ass7} or Assumption \ref{ass8}, we can also establish the relations between the first order and second order stationary points of \eqref{ob-r} and
the local saddle points for the particular case  of \eqref{obb}.

In this section, we consider the following problem
\begin{equation}\label{obb-3-1}
\begin{aligned}
\min_{\x\in\X}\max_{\y\in\Y} \,f(\x,\y):=c(\x,\y)
+&\lambda_1\sum_{i=1}^n\left( (g_{i}(\x_i)_+)^0+(g_{n+i}(\x_i)_+)^0\right)\\
-&\lambda_2\sum_{j=1}^m\left( (h_{j}(\y_j)_+)^0+(h_{m+j}(\y_j)_+)^0\right).
\end{aligned}\end{equation}

Problem \eqref{obb-3-1} is a case of  \eqref{obb} with  $\hat{n}=2n, \hat{m}=2m$ and
\begin{equation}\label{2n}
g_i(\x)=g_i(\x_i),\,\,  g_{n+i}(\x)= g_{n+i}(\x_i), \,\, i\in [n],
\end{equation}
\begin{equation}\label{2m}
 h_j(\y)=h_j(\y_j),\,\,  h_{m+j}(\y)= h_{m+j}(\y_j),\,\, j\in [m].
\end{equation}

Moreover, we  assume that $\X$ and $\Y$ satisfy the Slater's condition, i.e.
\begin{equation}\label{ass-X}
{\rm int}(\X)\neq\emptyset\quad\mbox{and}\quad {\rm int}(\Y)\neq\emptyset,
\end{equation}
and have the following structures
\begin{equation}\label{eq-cons-1}
{\mathcal{X}}:=\tilde{\mathcal{X}}\cap\bar{\mathcal{X}}\quad\mbox{and}\quad {\mathcal{Y}}:=\tilde{\mathcal{Y}}\cap\bar{\mathcal{Y}},
\end{equation}
where
\begin{equation}\label{eq-cons}
\tilde{\X}=\{{\x}\in \R^n:\underline{\u}\leq \x\leq \overline{{\u}}\},\quad \tilde{{\Y}}=\{{\y}\in \R^m:\underline{\v}\leq \y\leq \overline{\v}\}
\end{equation}
with $\underline{\u}=(\underline{{\u}}_1,\ldots,\underline{{\u}}_n)^\top$, $\overline{{\u}}=(\overline{\u}_1,\ldots,\overline{\u}_n)^\top\in {\R}^n$, $\underline{{\v}}=(\underline{{\v}}_1,\ldots,\underline{{\v}}_m)^\top$, $\overline{{\v}}=(\overline{\v}_1,\ldots,\overline{\v}_m)^\top\in {\R}^m$, $\underline{\u}<\overline{{\u}}$,
$\underline{{\v}}<\overline{{\v}}$, and
\begin{equation}\label{eqx2}
\bar{\X}=\{{\x}\in \R^n:\,u_t(\x)\leq0,\,t\in[\bar{n}]\},\quad \bar{{\Y}}=\{{\y}\in \R^m:v_s(\y)\leq0,\,s\in[\bar{m}]\}
\end{equation}
with Lipschitz continuous convex functions $u_t:\mathbb{R}^n\rightarrow\mathbb{R}$ for $t\in[\bar{n}]$ and
$v_s:\mathbb{R}^m\rightarrow\mathbb{R}$ for $s\in[\bar{m}]$.

Denote
$$\mathcal{T}_0(\x)=\{t\in[\bar{n}]:u_t(\x)=0\},\quad \mathcal{S}_0(\y)=\{s\in[\bar{m}]:v_s(\y)=0\}.$$

Since ${\rm int}(\bar{\mathcal{X}})\supseteq{\rm int}(\mathcal{X})\neq\emptyset$, by \cite[Theorems 6.8.2 and 6.8.3]{Pang-book2} and \cite[Proposition 5.3.1 and Remark 5.3.2]{H-Lema}, we have
\begin{equation}\label{normal}
\mathcal{N}_{\bar{\mathcal{X}}}(\x)=\sum\nolimits_{t\in \mathcal{T}_0(\x)}[0,+\infty)\partial u_t(\x), \quad \forall\, \x\in \bar{\mathcal{X}}.
\end{equation}
Using ${\rm int}(\mathcal{X})\neq\emptyset$ again, we have
\begin{equation}\label{eq3-3}
\mathcal{N}_{\mathcal{X}}(\x)=\mathcal{N}_{\tilde{\mathcal{X}}}(\x)+\mathcal{N}_{\bar{\mathcal{X}}}(\x).
\end{equation}
Similar calculation can be put forward to $\Y$.

Denote
$$\phi_{i}(\x_i)=(g_{i}(\x_i)_+)^0+(g_{n+i}(\x_i)_+)^0,\quad
\psi_{j}(\y_j)=(h_{j}(\y_j)_+)^0+(h_{m+j}(\y_j)_+)^0$$
and define their continuous relaxations by
\begin{equation}\label{eq-phi}
{\phi}_{i}^{R}(\x_i,\mu)=r(g_{i}(\x_i),\mu)+r(g_{n+i}(\x_i),\mu),\quad
{\psi}_{j}^{R}(\y_j,\mu)=r(h_{j}(\y_j),\mu)+r(h_{m+j}(\y_j),\mu)
\end{equation}
with $r$ in \eqref{c-re} and $\mu>0$. We consider the continuous relaxation of \eqref{obb-3-1} as follows
\begin{equation}\label{obb-3r}
\min_{\x\in\X}\max_{\y\in\Y} \,f^R(\x,\y,\mu):=c(\x,\y)+\lambda_1\sum_{i=1}^n {\phi}_{i}^{R}(\x_i,\mu)
-\lambda_2\sum_{j=1}^m{\psi}_{j}^{R}(\y_j,\mu).
\end{equation}

We impose the following assumption on functions $g$ and $h$ in (\ref{2n})-(\ref{2m})
 related to the sets
$\tilde{\X}$ and $\bar{\X}$ in (\ref{eq-cons})-(\ref{eqx2}).

\begin{assumption}\label{ass4}
There exist positive numbers $\tau$ and $\sigma$ such that the following conditions hold.
\begin{itemize}
\item [{\rm (i)}] For any $\x\in\X$, if there exists $i\in[{n}]$ such that $g_{i}(\x_{i})\in (0, \tau)$  (or $g_{n+i}(\x_{i})\in (0, \tau)$), then $\x_{i}\in{\rm int}(\tilde{\X}_{i})$ and
   \begin{align}
     &g_{n+i}(\x_i)\not\in[0,\tau] \;  (\mbox{or}\;g_{i}(\x_i)\not\in[0,\tau]), \; |g_{i}'(\x_{i})|\geq\sigma \; (\mbox{or}\;|g_{n+i}'(\x_{i})|\geq\sigma), \label{ass2-1}\\
     &\xi_{i}^t(\x)g_{i}'(\x_{i})\geq0 \; (\mbox{or}\;\xi_{i}^t(\x)g_{n+i}'(\x_{i})\geq0),\quad  \forall \xi^t(\x)\in\partial u_t(\x),\,t\in\mathcal{T}_0(\x).\label{ass2-3}
   \end{align}
\item [{\rm (ii)}] For any $\y\in\Y$, if there exists ${j}\in[m]$  such that $h_{j}(\y_{j}) \in (0, \tau)$ (or $h_{m+j}(\y_{j})\in (0, \tau)$), then $\y_{j}\in{\rm int}(\tilde{\Y}_{j})$ and
    \begin{align}
     &   h_{m+j}(\y_{j})\not\in[0,\tau] \;(\mbox{or}\;h_{j}(\y_{j})\not\in[0,\tau]),\; |h_{j}'(\y_{j})|\geq\sigma \; (\mbox{or}\;| h_{m+j}'(\y_{j})|\geq\sigma),\label{ass2-2}\\
     &\eta_{j}^s(\y)h_{j}'(\y_{j})\geq 0 \; (\mbox{or}\;\eta_{j}^s(\y)h_{m+j}'(\y_{j})\geq0), \quad \forall \eta^s(\y)\in\partial v_s(\y),\,s\in\mathcal{S}_0(\y).\label{ass2-4}
   \end{align}
\end{itemize}
\end{assumption}
Under Assumption \ref{ass4}-(i), if there exists $i\in[{n}]$ such that $g_{i}(\x_{i})\in (0, \tau)$  (or $g_{n+i}(\x_{i})\in (0, \tau)$), by \eqref{normal} and \eqref{eq3-3}, then $[\mathcal{N}_{{\mathcal{X}}}(\x)]_i=0$ or $[\mathcal{N}_{{\mathcal{X}}}(\x)]_{{i}}=[\mathcal{N}_{\bar{\mathcal{X}}}(\x)]_{{i}}=[\sum\nolimits_{t\in \mathcal{T}_0(\x)}[0,+\infty)\partial u_t(\x)]_{{i}}$, which together with \eqref{ass2-3} implies that
\begin{equation}\label{eq3-1}
\xi_{{i}}(\x)g_{{i}}'(\x_{{i}})\geq0\,(\mbox{or}\;\xi_{{i}}(\x)g_{n+{i}}'(\x_{{i}})\geq0),\quad \forall\,  \xi(\x)\in \mathcal{N}_{{\mathcal{\X}}}(\x).
\end{equation}
Similarly, Assumption \ref{ass4}-(ii) implies that if there exists ${j}\in[m]$  such that $h_{j}(\y_{j}) \in (0, \tau)$ (or $h_{m+j}(\y_{j})\in (0, \tau)$), then
\begin{equation}\label{eq3-2}
\eta_{{j}}(\y) h'_{{j}}(\y_{{j}})\geq0\;(\mbox{or}\;\eta_{{j}}(\y) h'_{m+{j}}(\y_{{j}})\geq0),\quad \forall \eta(\y)\in \mathcal{N}_{{\mathcal{\Y}}}(\y).
\end{equation}

When $g_{n+i}(\x_i)=-1$ and $h_{m+j}(\y_j)=-1$ for all $i\in[n]$ and $j\in[m]$ in \eqref{obb-3-1}, then \eqref{obb-3-1} is reduced to
\begin{equation}\label{eqad1}
\min_{\x\in\X}\max_{\y\in\Y} \,f(\x,\y):=c(\x,\y)
+\lambda_1\sum_{i=1}^n(g_{i}(\x_i)_+)^0
-\lambda_2\sum_{j=1}^m(h_{j}(\y_j)_+)^0.
\end{equation}

\begin{remark}\label{remark1}
Consider
\begin{equation}\label{eq-g1}
f(\x,\y):=c(\x,\y)+\lambda_1\|(\x-\underline{\a})_+\|_0+
\lambda_1\|(\overline{\a}-\x)_+\|_0
-\lambda_2\|(\y-\underline{\b})_+\|_0-\lambda_2\|(\overline{\b}-\y)_+\|_0,
\end{equation}
with $\underline{\a}$, $\overline{\a}\in\R^n$ and $\underline{\b}$, $\overline{\b}\in\R^m$. Then
\eqref{eq-g1} is a special case of problem \eqref{obb-3-1} with
$$
g_{i}(\x)=\x_i-{\underline{\a}}_i,\, g_{n+i}(\x)=\overline{\a}_{i}-\x_i,\, i\in[n];\;
h_{j}(\y)=\y_j-\underline{\b}_j,\, h_{m+j}(\y)=\overline{\b}_{j}-\y_j,\, j\in[m].
$$

In particular, the following three cases satisfy Assumption \ref{ass4}-(i), while
the judgment on Assumption \ref{ass4}-(ii) is the same.
\begin{itemize}
\item {\rm Case 1:} Let $\X=\{{\x}\in \R^n:\underline{\u}\leq \x\leq \overline{{\u}}\}$ and $f$ is defined as in \eqref{eq-g1} with $\overline{\a},\,\underline{\a}\in\X$. Then,
Assumption \ref{ass4}-(i) holds with
\begin{equation}\label{eq-g2}
\sigma=1\quad\mbox{and}\quad \tau=\min\{\tau_{\x,1},\tau_{\x,2}\},
\end{equation}
where
\begin{eqnarray*}
\begin{aligned}
&\tau_{\x,1}=\min\{1,\overline{\u}_i-\underline{\a}_i,
\overline{\a}_{i}-\underline{{\u}}_i:\overline{\u}_i>\underline{\a}_i,
\overline{\a}_{i}>\underline{{\u}}_i,i\in[n]\},\\ &\tau_{\x,2}=\min\{1,|\underline{\a}_i-\overline{\a}_{i}|/2:
\underline{\a}_i\neq\overline{\a}_{i},\,i\in[n]\}.
\end{aligned}\end{eqnarray*}
\item {\rm Case 2:} Let $f$ be defined as in \eqref{eq-g1} with $\underline{\a}\geq{\bf 0}$, $\overline{\a}\leq{\bf 0}$,  $\underline{\b}\geq{\bf 0}$, $\overline{\b}\leq{\bf 0}$,
$\bar{\X}=\{{\x}\in \R^n:\underline{\u}\leq \x\leq \overline{{\u}}\}$ and
$\tilde{\X}$ be
$$\tilde{\mathcal{X}}=\{\x:\,\|\x\|_2\leq \delta\}\quad\mbox{or}
\quad \tilde{\mathcal{X}}=\{\x:\,\|\x\|_1\leq \delta\}$$
with $\delta>0$. Then Assumption \ref{ass4}-(i) also holds with $\sigma$ and $\tau$ in \eqref{eq-g2}.
\item {\rm Case 3:} Let $\bar{\X}=\{{\x}\in \R^n:\underline{\u}\leq \x\leq \overline{{\u}}\}$,
$\tilde{\X}=\{\x:\textbf{A}\x\leq \c\}$ with $\textbf{A}\in\R_{+}^{r_1\times n}$, $\c\in\R^{r_1}$ and $f$ be specialized to
$$
f(\x,\y):=c(\x,\y)+\lambda_1\|(\x-\underline{\a})_+\|_0
-\lambda_2\|(\y-\underline{\b})_+\|_0$$
with $\underline{\a}_i\in[\underline{\u}_i,\overline{\u}_i)$ and $\underline{\b}_j\in[\underline{\v}_j,\overline{\v}_j)$ for $i\in[n]$ and $j\in[m]$. Then Assumption \ref{ass4}-(i) also holds with $\sigma=1$ and $\tau=\min\{1,\overline{\u}_i-\underline{\a}_i
:\overline{\u}_i>\underline{\a}_i,i\in[n]\}$.
\end{itemize}
\end{remark}

In particular, case 2 in Remark \ref{remark1} indicates that problem \eqref{obb-4}
with $\X=\{\x:\underline{\u}\leq \x\leq \overline{{\u}},\,\|\x\|_1\leq1\}$ and $\Y=\{\y:\underline{\v}\leq \x\leq \overline{{\v}},\,\|\y\|_1\leq1\}$ satisfies Assumption \ref{ass4}, and the problems in case 3 satisfying Assumption \ref{ass4} include
$$\min_{\x\in\X}\max_{\y\in\Y} \,f(\x,\y):=c(\x,\y)+\lambda_1\|\x_+\|_0
-\lambda_2\|\y_+\|_0$$
with $\X=\{\x:-\e\leq \x\leq \e,\,{\e}^\top\x\leq1\}$ and $\Y=\{\y:-\e\leq \y\leq \e,\,{\e}^\top\y\leq1\}$ as a special case.
\subsection{Density function $\rho$ under Assumption \ref{ass7}}\label{section5.1}
In this subsection, we show that Assumptions \ref{ass6} and \ref{ass2} hold for
problems \eqref{obb-3-1} and {its continuous relaxation} \eqref{obb-3r} when the density function $\rho$ satisfies Assumption \ref{ass7}. Moreover, when the density function $\rho$ is as in Example \ref{example4.1}, which satisfies Assumption \ref{ass7}, we will give more discussion on the weak-d stationary points of \eqref{obb-3r} and its smoothing version in \eqref{ob-cr}.

\subsubsection{Relations on saddle points and minimax points}
From the boundedness of $\X$ and $\Y$,
there exists a positive constant $L_{c,1}$ such that for all $\x\in\X$ and $\y\in\Y$, it holds
\begin{equation}\label{eq-c}
\|\partial_{\x}c(\x,\y)\|_{\infty}\leq L_{c,1}\quad \mbox{and}\quad
\|\partial_{\y}c(\x,\y)\|_{\infty}\leq L_{c,1}.
\end{equation}

For a given $\mu\in\R_{++}$ and $\bar{\y}\in\Y$, if $\x^*$ is a local solution of $\min_{\x\in\X}f^R(\x,\bar{\y},\mu)$, then
\begin{equation}\label{first1-1}
\0\in\partial_{\x}f^R(\x^*,\bar{\y},\mu)+\mathcal{N}_{\X}(\x^*).
\end{equation}
Similarly, for $\bar{\x}\in\X$, if $\y^*$ is a local solution of $\max_{\y\in\Y}f^R(\bar{\x},\y,\mu)$, then
\begin{equation}\label{first2-1}
\0\in-\partial_{\y}f^R(\bar{\x},\y^*,\mu)+\mathcal{N}_{\Y}(\y^*).
\end{equation}

For $i\in[n]$ and $j\in[{m}]$, by \cite[Proposition 2.3.9]{Clarke}, we have
\begin{equation}\label{eq-r11}\begin{aligned}
&\partial_{\x_i} \phi_{i}^R(\x_i,\mu)\subseteq\tilde{\partial}_{\x_i}\phi_{i}^R(\x_i,\mu):=
\partial_{t} r(t,\mu)_{t=g_{i}(\x_i)}g_{i}'(\x_i)+\partial_{t} r(t,\mu)_{t=g_{n+i}(\x_i)}g_{n+i}'(\x_i),\\
&\partial_{\y_j} \psi_{j}^R(\y_j,\mu)\subseteq\tilde{\partial}_{\y_j}\psi_{j}^R(\y_j,\mu):=\partial_{t} r(t,\mu)_{t=h_{j}(\y_j)}h_{j}'(\y_j)+\partial_{t} r(t,\mu)_{t=h_{m+j}(\y_j)}h_{m+j}'(\y_j).
\end{aligned}\end{equation}

By \cite[Corollary 2]{Clarke}
and recalling \eqref{eq-r11}, one has
\begin{eqnarray}
&\partial_{\x}f^R(\x^*,\bar{\y},\mu)\subseteq \tilde{\partial}_{\x}f^R(\x^*,\bar{\y},\mu):= \partial_{\x}c(\x^*,\bar{\y})+\lambda_1\sum_{i=1}^n\tilde{\partial}_{\x_i}{\phi}_{i}^{R}(\x^*_{i},\mu)\e_i,\label{first1-2}\\
&\partial_{\y}f^R(\bar{\x},\y^*,\mu)\subseteq\tilde{\partial}_{\y}f^R(\bar{\x},\y^*,\mu):= \partial_{\y}c(\bar{\x},\y^*)-\lambda_2\sum_{j=1}^m\tilde{\partial}_{\y_j}
{\psi}_{j}^R(\y^*_j,\mu)\e_j.\label{first2-2}
\end{eqnarray}
Combining \eqref{first1-1} with \eqref{first1-2}, and \eqref{first2-1} with \eqref{first2-2}, we obtain that
\begin{itemize}
\item if $\x^*$ is a local solution of $\min_{\x\in\X}f^R(\x,\bar{\y},\mu)$, then
\begin{equation}\label{first1}
\0\in\tilde{\partial}_{\x}f^R(\x^*,\bar{\y},\mu)+\mathcal{N}_{\X}(\x^*);
\end{equation}
\item if $\y^*$ is a local solution of $\max_{\y\in\Y}f^R(\bar{\x},\y,\mu)$, then
\begin{equation}\label{first2}
\0\in-\tilde{\partial}_{\y}f^R(\bar{\x},\y^*,\mu)+\mathcal{N}_{\Y}(\y^*).
\end{equation}
\end{itemize}

If $(\x^*,\y^*)\in\X\times\Y$ is a local saddle point of \eqref{obb-3r}, then \eqref{first1} and \eqref{first2} hold at $\bar{\x}=\x^*$ and $\bar{\y}=\y^*$.

In the rest of this paper, we denote ${\lambda}=\min\{\lambda_1,\lambda_2\}$ and set
\begin{equation}\label{eq-mu}
\bar{\mu}_1=\min\left\{\frac{\tau}{\alpha},\frac{{\lambda}\sigma\underline{\rho}}{L_{c,1}}
\right\}
\end{equation}
with $\tau$, $\sigma$ in Assumption \ref{ass4}, $\alpha$ in \eqref{eq-ass1}, $\underline{\rho}$ in Assumption \ref{ass7}, and $L_{c,1}$ a constant satisfying the two inequalities in \eqref{eq-c}.
Here, $0<\mu<\bar{\mu}_1\leq \tau/\alpha$ gives $\mu\alpha<\tau$, which together with \eqref{eq-rho} implies
\begin{align*}
&\mbox{if $g_{\hat{i}}(\x_i)\not\in(0,\tau)$, then $r(g_{\hat{i}}(\x_i),\mu)=(g_{\hat{i}}(\x_i)_+)^0$, $\forall   i\in[{n}]$, $\hat{i}\in\{i,n+i\}$}; \\
&\mbox{if $h_{\hat{j}}(\y_j)\not\in(0,\tau)$, then $r(h_{\hat{j}}(\y_j),\mu)=(h_{\hat{j}}(\y_j)_+)^0$, $\forall j\in[{m}]$, $\hat{j}\in\{j,m+j\}$}.
\end{align*}
Next, we will derive the lower bounds in \eqref{eq-ass6-1} and \eqref{eq-ass6-2} based on \eqref{first1} and \eqref{first2}, respectively.
\begin{proposition}\label{prop2}
Suppose problem \eqref{obb-3-1} satisfies Assumption \ref{ass4}. When density function $\rho$ satisfies Assumption \ref{ass7} and $0<\mu<\bar{\mu}_1$ with $\bar{\mu}_1$ defined in \eqref{eq-mu}, then the continuous relaxation model in \eqref{obb-3r} owns the following properties for any $i\in[n]$ and $j\in[m]$:
\begin{eqnarray}
\hspace{-0.3cm}&\mbox{if \eqref{first1} holds at $({\x}^*,\bar{\y})\in\X\times\Y$, then
$g_{\hat{i}}({\x}^*_i)\not\in(0,\alpha\mu)$, $\forall \,\hat{i}\in\{i,n+i\}$,}\label{lb1-1}\\
\hspace{-0.3cm}&\mbox{if \eqref{first2} holds at $(\bar{\x},{\y}^*)\in\X\times\Y$, then
$h_{\hat{j}}({\y}^*_j)\not\in(0,\alpha\mu)$,
$\forall \, \hat{j}\in\{j,m+j\}$}.\label{lb1-2}
\end{eqnarray}
\end{proposition}
\begin{proof}
We argue the above statements by contradiction.

If there exists $\tilde{i}\in[{n}]$ such that $0<g_{\tilde{i}}(\x^*_{\tilde{i}})<\alpha\mu$, by $\alpha\mu<\alpha\bar{\mu}_1\leq\tau$, \eqref{ass2-1} in Assumption \ref{ass4}, \eqref{eq-r11} and \eqref{first1-2}, we obtain
$\x_{\tilde{i}}^*\in{\rm int}(\tilde{\X}_{\tilde{i}})$, and
$$\begin{aligned}
&{[\tilde{\partial}_{\x}f^R(\x^*,\bar{\y},\mu)]}_{\tilde{i}}
=[\partial_{\x}c(\x^*,\bar{\y})]_{\tilde{i}}
+\lambda_1\partial_{t} r(t,\mu)_{t=g_{\tilde{i}}(\x^*_{\tilde{i}})} g_{\tilde{i}}'(\x^*_{\tilde{i}}).
\end{aligned}$$

From  \eqref{eq-19}, Assumption \ref{ass4}-(i), \eqref{eq3-1} and \eqref{first1}, we further have that there exists $\omega\in \partial_{t} r(t,\mu)_{t=g_{\tilde{i}}(\x^*_{\tilde{i}})}$ such that
$$\lambda_1\sigma\underline{\rho}/\mu
\leq\lambda_1|\omega g_{\tilde{i}}'(\x^*_{\tilde{i}})|\leq L_{c,1},$$
which  contradicts $\mu<\bar{\mu}_1\leq\lambda_1\sigma\underline{\rho}/L_{c,1}$. Thus, \eqref{lb1-1} holds.
Similar discussion can be derived if there exists $\tilde{i}\in[{n}]$ such that $0<g_{n+\tilde{i}}(\x^*_{\tilde{i}})<\alpha\mu$ and to \eqref{lb1-2}.
\end{proof}

Thus, under the conditions in Proposition \ref{prop2}, Assumption \ref{ass6} holds naturally for any $\mu\in(0,\bar{\mu}_1)$ with $\bar{\mu}_1$ defined in \eqref{eq-mu}.
In what follows, we will verify Assumption \ref{ass2} in this situation. To proceed, we need the following exact penalty result, which holds by \eqref{ass-X} and the boundedness of $\X$.
\begin{proposition}\label{prop7}
For any $\mu>0$, $\x^*\in\X$ and $\delta>0$,  there exists a $\beta>0$ such that any global minimizer of $\max_{\y'\in\Y\cap\mathbf{B}(\y^*,\pi(\delta))}f^R({\x},{\y}',\mu)$ on $\X\cap\mathbf{B}(\x^*,\delta)$ is a global minimizer of $p(\x,\mu)$ on $\tilde{\X}\cap\mathbf{B}(\x^*,\delta)$, where
$$p(\x,\mu)=\max_{\y'\in\Y\cap\mathbf{B}(\y^*,\pi(\delta))}f^R({\x},{\y}',\mu)+\beta\sum\nolimits_{t\in[\bar{n}]}u_t(\x)_+,$$
where $\pi(\cdot)$ is a function defined as in Definition \ref{defn-min}.
\end{proposition}
\begin{proposition}\label{prop3}
Suppose problem \eqref{obb-3-1} satisfies Assumption \ref{ass4}.
Then Assumption \ref{ass2} holds for \eqref{obb-3r} when density function $\rho$ satisfies Assumption \ref{ass7} and $0<\mu<\bar{\mu}_1$ with $\bar{\mu}_1$ in \eqref{eq-mu}. Moreover, all global minimax points and local minimax points of \eqref{obb-3r} satisfy the lower bounds in \eqref{eq-ass6-1} and \eqref{eq-ass6-2}.
\end{proposition}
\begin{proof}
For an $\bar{\x}\in\X$, suppose ${\y^*}$ is a global maximizer of $f^R(\bar{\x},\y,\mu)$ on $\Y$, then \eqref{first2} holds, by Proposition \ref{prop2}, which implies Assumption \ref{ass2}-(ii) holds, i.e.
$h_{j}({\y^*_j})\not\in(0,\alpha\mu)$ and $h_{m+j}({\y^*_j})\not\in(0,\alpha\mu)$,
$\forall j\in[{m}]$. Then, we only need to prove the condition (i) in Assumption \ref{ass2}.

For any $\x\in\tilde{\X}$, denote $\y_{\x}$ the maximizer of $f^R(\x,\y,\mu)$ on $\Y$, and
$\x^*$ a global minimizer of $\max_{\y'\in\Y}f^R({\x},{\y}',\mu)$ on $\X$. By Proposition \ref{prop7} with a sufficient large $\delta$ such that ${\X}\subseteq\mathbf{B}(\x^*,\delta)$, $\Y\subseteq\mathbf{B}(\y^*,\delta)$ and $\Y\subseteq\mathbf{B}(\y^*,\pi(\delta))$,
it gives the existence of $\beta$ such that, for any $\x\in\tilde{\X}\cap\mathbf{B}(\x^*,\delta)$, it holds
$$
f^R({\x}^*,{\y}_{\x},\mu)\leq \max_{\y'\in\Y}f^R({\x}^*,{\y}',\mu)=p(\x^*,\mu)\leq
f^R(\x,\y_{\x},\mu)+\beta\sum\nolimits_{t\in[\bar{n}]}u_t(\x)_+.
$$
Then, for any $\x\in\tilde{\X}\cap\mathbf{B}(\x^*,\delta)$, it has
\begin{equation}\label{eq2.10}
0\leq
c(\x,\y_{\x})-c({\x}^*,\y_{\x})+\lambda_1\sum_{i=1}^n(\phi_{i}^R(\x_i,\mu)
-\phi_{i}^R({\x}^*_i,\mu))
+\beta\sum_{t\in[\bar{n}]}u_t(\x)_+.
\end{equation}

To prove the lower bound in \eqref{eq-ass6-1}, we also argue it by contradiction.
Assume there exists $\tilde{i}\in[{n}]$ such that $0<g_{\tilde{i}}({\x}^*_{\tilde{i}})<\alpha\mu$, then
by Assumption \ref{ass4},
it holds
\begin{equation}\label{eq2.5}
\x_{\tilde{i}}^*\in{\rm int}(\tilde{\X}_{\tilde{i}}),\quad g_{n+\tilde{i}}(\x_{\tilde{i}}^*)\not\in[0,\tau]\quad\mbox{and}\quad|g_{\tilde{i}}'(\x_{\tilde{i}}^*)|\geq\sigma.
\end{equation}
Without loss of generality, we first assume $g_{\tilde{i}}'(\x_{\tilde{i}}^*)\geq\sigma$.

By $\mu<\bar{\mu}_1$, it has $\lambda_1\underline{\rho}\sigma-\mu L_{c,1}>0$. Then,
let $\epsilon$ be a sufficiently small positive number such that
\begin{equation}\label{eq-15}
(\lambda_1\underline{\rho}+\beta\bar{n}\mu)\epsilon<\lambda_1\underline{\rho}\sigma-\mu L_{c,1},
\end{equation}
which implies $0<\epsilon<\sigma$.

By \eqref{eq2.5} and the continuously differentiability of $g_{i}$ and $g_{n+i}$,
there exists a $\delta_1>0$ such that for any $\x\in \mathbf{B}(\x^*,\delta_1)$, it holds \begin{equation}\label{eq15}
\x_{\tilde{i}}\in{\rm int}(\tilde{\X}_{\tilde{i}}), \,\, g_{n+\tilde{i}}(\x_{\tilde{i}})\not\in[0,\tau],\,\, g_{\tilde{i}}(\x_{\tilde{i}})\in(0,\alpha\mu)\quad\mbox{and}\quad
g_{\tilde{i}}'(\x_{\tilde{i}})\geq\sigma-\epsilon.
\end{equation}

Moreover, using the Lipschitz continuity of $u_t$ and $\x^*\in\bar{\X}$, there exists $\delta_2\in(0,\delta_1]$ such that
\begin{equation}\label{eq2.6}
\mathcal{T}_+({\x})=\{t\in[\bar{n}]:u_t({\x})>0\}\subseteq \mathcal{T}_0({\x}^*), \,\forall\, \x\in \mathbf{B}(\x^*,\delta_2).
\end{equation}
Recalling \eqref{ass2-3} and the positiveness of $g_{\tilde{i}}'(\x_{\tilde{i}}^*)$, it gives $\xi^{t}_{\tilde{i}}(\x^*)\geq0$, $\forall \xi^{t}(\x^*)\in \partial u_t(\x^*)$, $t\in\mathcal{T}_0({\x}^*)$. By the upper semicontinuity of $\partial u_t$, it implies the existence of $\delta_3\in(0,\delta_2]$ such that
\begin{equation}\label{eq2.15}
\xi^{t}_{\tilde{i}}(\x)\geq-\epsilon,\quad\,\forall \xi^{t}(\x)\in \partial u_t(\x),\,\x\in \mathbf{B}(\x^*,\delta_3),\,t\in\mathcal{T}_0(\x^*).
\end{equation}

Based on the preliminary settings,
we choose $\x$ to be $\bar{\x}$ in \eqref{eq2.10}, where $\bar{\x}_i={\x}_i^*$ for $i\neq\tilde{{i}}$ and $\bar{\x}_{\tilde{{i}}}={\x}^*_{\tilde{i}}-{\delta}_3/{2}$,
then $\bar{\x}\in \mathbf{B}(\x^*,\delta_3)\cap\tilde{
\X}$ and
\begin{equation}\label{eq2.10-2}
\begin{aligned}
\sum_{i=1}^n(\phi_{i}^R(\bar{\x}_i,\mu)
-\phi_{i}^R({\x}^*_i,\mu))
=&r(g_{\tilde{i}}(\bar{\x}_{\tilde{i}}),\mu)-r(g_{\tilde{i}}({\x}_{\tilde{i}}^*),\mu)\\
=&\xi(\tilde{\x}_{\tilde{i}}) g_{\tilde{i}}'(\tilde{\x})(\bar{\x}_{\tilde{i}}-\x^*_{\tilde{i}})
\leq-\underline{\rho}(\sigma-\epsilon){\delta}_3/({2}\mu),
\end{aligned}\end{equation}
where the
second equality uses the mean value theorem
with $\xi(\tilde{\x}_{\tilde{i}})\in\partial_{t} r(t,\mu)_{t=g_{\tilde{i}}(\tilde{\x}_{\tilde{i}})}$ and $\tilde{\x}\in{\rm co}\{\bar{\x},\x^*\}\subseteq\mathbf{B}(\x^*,\delta_3)$, the last inequality follows by the definition of $\bar{\x}$, \eqref{eq15} and Assumption \ref{ass7}.

Next,
using \eqref{eq2.6} and \eqref{eq2.15}, we have
\begin{equation}\label{eq2.10-2-2}
\begin{aligned}
\sum\nolimits_{t\in[\bar{n}]}u_t(\bar{\x})_+
=\sum\nolimits_{t\in\mathcal{T}_+(\bar{\x})}u_t(\bar{\x})
=&\sum\nolimits_{t\in\mathcal{T}_+(\bar{\x})}(u_t(\bar{\x})-u_t({\x}^*))\\
=&\sum\nolimits_{t\in\mathcal{T}_+(\bar{\x})}\xi^t(\check{{\x}})^{\rm T}(\bar{\x}-{\x}^*)
\leq\delta_3\bar{n}\epsilon/2,
\end{aligned}\end{equation}
where $\xi^t(\check{{\x}})\in\partial u_t(\check{{\x}})$ with $\check{{\x}}\in{\rm co}\{\bar{\x},\x^*\}\subseteq\mathbf{B}(\x^*,\delta_3)\cap\tilde{\X}$.
Substituting \eqref{eq2.10-2} and \eqref{eq2.10-2-2} into \eqref{eq2.10}, one has
\begin{equation}\label{eq-e12}
0
\leq L_{c,1}-\lambda_1\underline{\rho}(\sigma-\epsilon)/\mu+ \beta\bar{n}\epsilon,
\end{equation}
which contradicts the value of $\epsilon$ in \eqref{eq-15}.

Similar to the above discussion, when $g_{\tilde{i}}'({\x}^*_{\tilde{i}})\leq-\sigma$, we can choose $\x$ to be $\bar{\x}$ in \eqref{eq2.10} with $\bar{\x}_{i}={\x}^*_{i}$ for $i\neq\tilde{{i}}$ and $\bar{\x}_{{\tilde{i}}}={\x}^*_{{\tilde{i}}}+\delta_3/2$. Then we can also obtain \eqref{eq-e12} and the same contradiction. If there exists $\tilde{i}\in[{n}]$ such that $0<g_{n+\tilde{i}}({\x}^*_{\tilde{i}})<\alpha\mu$, the analysis is almost same.
Thus, for all $i\in[{n}]$, $g_{i}({\x}_i^*)\not\in(0,\alpha\mu)$ and $g_{n+i}({\x}_i^*)\not\in(0,\alpha\mu)$, which give (i) in Assumption \ref{ass2}. Therefore, Assumption \ref{ass2} holds, which implies that any global minimax point of \eqref{obb-3r} satisfies the lower bounds in \eqref{eq-ass6-1} and \eqref{eq-ass6-2}.

If $({\x}^*,{\y}^*)$ is a local minimax point of \eqref{obb-3r} with positive constant $\delta_0$ and function $\pi(\delta)$ in Definition \ref{defn-min},
the above discussion can also be followed up when we restrict $\x$ to $\mathbf{B}({\x}^*,\delta)$, $\y$ to $\mathbf{B}({\y}^*,\delta)$ and
$\mathbf{B}({\y}^*,\pi(\delta))$ in the corresponding places.
Thus, any local minimax point of problem \eqref{obb-3r}
satisfies the lower bounds in \eqref{eq-ass6-1} and \eqref{eq-ass6-2}.
\end{proof}

Similar to the proof of Theorem \ref{theorem3}, we have the relation on the local minimax points between problems \eqref{obb-3-1} and \eqref{obb-3r}. Combining this with the above discussion, we conclude the relations on \eqref{obb-3-1} and \eqref{obb-3r} in the following theorem.
\begin{theorem}\label{theorem5}
Suppose problem \eqref{obb-3-1} satisfies Assumption \ref{ass4}, density function $\rho$ satisfies Assumption \ref{ass7} and $0<\mu<\bar{\mu}_1$ with $\bar{\mu}_1$ defined in \eqref{eq-mu},
then the following statements hold:
\begin{itemize}
\item [{\rm (i)}] $(\x^*,\y^*)$ is a saddle point (global minimax point) of problem \eqref{obb-3-1} if and only if it is a
saddle point (global minimax point) of \eqref{obb-3r};
\item [{\rm (ii)}] $(\x^*,\y^*)$ is an $\alpha\mu$-strong
local saddle point (local minimax point) of \eqref{obb-3-1}, if it is a local saddle point (local minimax point) of problem \eqref{obb-3r}.
\end{itemize}
\end{theorem}

\subsubsection{Stationary points of \eqref{obb-3r} with $\rho$ in Example \ref{example4.1}}\label{section4.4}
In this subsection, we focus on the relations of \eqref{obb-3-1} and \eqref{obb-3r} when the density function $\rho$ is defined as in Example \ref{example4.1}, which makes the corresponding function $r$ satisfy Assumption \ref{ass7} with $\alpha=\underline{\rho}=\overline{\rho}=1$. Theorem \ref{theorem5} has established the relations on the (local) saddle points and (global or local) minimax points between problems \eqref{obb-3-1} and \eqref{obb-3r}.
In this subsection, we suppose problem \eqref{obb-3-1} satisfies Assumption \ref{ass4}, and functions $g_{i}$, $h_{j}$ in \eqref{obb-3-1} are convex for all $i\in[{2n}]$ and $j\in[2{m}]$. We will study the relations on a class of stationary points of \eqref{obb-3r} with the $\mu$-strong local saddle points of \eqref{obb-3-1} in what follows.

For a locally Lipschitz continuous function $\varphi:\R^n\rightarrow\R$, the generalized (Clarke) directional derivative \cite{Clarke} of $\varphi$ at point $\x$ in direction $\v$ is well-defined, i.e.
$$\varphi^{\circ}(\x,\v)=\limsup_{\z\rightarrow\x,t\downarrow0}\frac{\varphi(\z+t\v)-\varphi(\z)}{t}.$$
Function $\varphi$ is said to be Bouligand-differentiable (B-differentiable) at $\x$, if $\varphi$ is locally Lipschitz continuous around $\x$ and directionally differentiable at $\x$, i.e. for any $\v\in\mathbb{R}^n$,
$$\varphi'(\x,\v)=\limsup_{t\downarrow0}\frac{\varphi(\x+t\v)-\varphi(\x)}{t}\quad\mbox{exists}.$$
It is well-known that $\varphi^{\circ}(\x,\v)\geq \varphi'(\x,\v)$ in general and these two directional derivatives are the same if function $\varphi$ is (Clarke) regular \cite{Clarke}. However, most nonconvex functions are not regular and a nonsmooth nonconvex function is not always directionally differentiable.
Notice that convex functions and differentiable functions are directionally differentiable, then a DC (difference-of-convex) function is directionally differentiable \cite{Rock-Wets}, where we call function $\varphi$ a DC function, if it can be formulated by the difference of two convex functions. This promotes some kinds of stationary points for the DC programming \cite{Pang2017}, such as the d(irectional)-stationary point and the weak d-stationary point, both of which are generally stronger than the Clarke stationary point.

Note that $r(t,\mu)$ in Example \ref{example4.1} can be expressed by the following DC function
\begin{equation}\label{eq-r1}
r(t,\mu)=t_+/\mu-(t-\mu)_+/\mu.
\end{equation}
From the definitions of $\phi_{i}^R$ and $\psi_{j}^R$ in \eqref{eq-phi}, the objective function in \eqref{obb-3r} has the formulation of
\begin{equation}\label{eqfR}
\begin{aligned}
f^R(\x,\y,\mu)=c(\x,\y)+&\lambda_1\sum_{i=1}^n
\left(g_{i}(\x_i)_+/\mu-(g_{i}(\x_i)-\mu)_+/\mu\right)\\
+&\lambda_1\sum_{i=1}^n\left(g_{n+i}(\x_i)_+/\mu
-(g_{n+i}(\x_i)-\mu)_+/\mu\right)\\
-&\lambda_2\sum_{j=1}^m \left(h_{j}(\y_j)_+/\mu-(h_{j}(\y_j)-\mu)_+/\mu\right)\\
-&\lambda_2\sum_{j=1}^m\left(h_{m+j}(\y_j)_+/\mu-
(h_{m+j}(\y_j)-\mu)_+/\mu\right).
\end{aligned}\end{equation}

For fixed ${\x}^*\in\mathcal{X}$, ${\y}^*\in\mathcal{Y}$ and $\mu\in\R_{++}$,
denote
\begin{eqnarray*}
f_{{\y}^*,\mu}^R(\x)\triangleq f^R(\x,{\y}^*,\mu),\quad f_{{\x}^*,\mu}^R(\y)\triangleq f^R({\x}^*,{\y},\mu),
\end{eqnarray*}
and consider the following two optimization problems
\begin{equation}\label{eq2.21}
\min_{\x\in\X}f^R_{{\y}^*,\mu}(\x)\quad\mbox{and}\quad \max_{\y\in\Y}f^R_{{\x}^*,\mu}(\y).
\end{equation}
By \eqref{eqfR} and the convexity-concavity of $c$, the two objective functions in \eqref{eq2.21}
are DC functions with respect to $\x$ and $\y$, and then they are B-differentiable on $\X$ and $\Y$, respectively.
For the sake of completeness, we recall the definition of d-stationary point in DC programming. We call
$\x^*\in \X$ a d-stationary point \cite[Definition 6.1.1]{Pang-book3} of the minimization problem in \eqref{eq2.21}, if
\begin{equation}\label{d-s1}
(f_{{\y}^*,\mu}^R)'(\x^*;\x-\x^*)\geq0,\quad\forall\, \x\in\X,
\end{equation}
which is a necessary optimality condition to the minimization program in \eqref{eq2.21}.

Define
$$
\varpi_{1}(t)=t, \quad\varpi_{2}(t)=0,\quad \varpi(t)=\max\{\varpi_{1}(t),\varpi_{2}(t)\}$$
$$\mbox{and}\quad \mathcal{D}(t)=\{d\in\{1,2\}:\varpi(t)=\varpi_d(t)\}.$$
It is clear that
$$\mbox{$\varpi_{1}'(t)=1$, $\varpi_{2}'(t)=0$ and $\partial \varpi(t)=\left\{\begin{aligned}
&1&&\mbox{if $t>0$}\\
&[0,1]&&\mbox{if $t=0$}\\
&0&&\mbox{if $t<0$}.
\end{aligned}\right.$}$$

\eqref{d-s1} is equivalent to that, for any $q_{i}^{*}\in\mathcal{D}(g_{i}(\x^*_i)-\mu)$ and
$q_{n+i}^{*}\in\mathcal{D}(g_{n+i}(\x^*_i)-\mu)$, it holds
\begin{equation}\label{eq2.25-2}
\begin{aligned}
&\frac{\lambda_1}{\mu}\sum_{i=1}^n\varpi_{q_{i}^{*}}'(t)_{t=g_{i}(\x^*_i)-\mu}
g_{i}'(\x^*_i)\e_i+\frac{\lambda_1}{\mu}\sum_{i=1}^n\varpi_{q_{n+i}^{*}}'(t)_{t=g_{n+i}(\x^*_i)-\mu}
g_{n+i}'(\x^*_i)\e_i\\
\in&\partial_{\x} c(\x^*,\y^*)+
\frac{\lambda_1}{\mu}\partial\left(\sum_{i=1}^n\varpi(g_{i}(\x^*_i))+\varpi(g_{n+i}(\x^*_i))\right)
+\mathcal{N}_{\X}(\x^*),
\end{aligned}
\end{equation}
in which by \cite[Proposition 2.3.10]{Clarke},
$$\partial\left(\sum_{i}\varpi(g_i(\x^*))+\varpi(g_{n+i}(\x^*_i))\right)=\sum_{i}
\left(\partial\varpi(t)_{t=g_{i}(\x^*_i)} g'_{i}(\x^*_i)+\partial\varpi(t)_{t=g_{n+i}(\x^*_i)} g'_{n+i}(\x^*_i)\right)\e_i.$$
Similarly, we call $\y^*\in \Y$ a d-stationary point of the maximization problem in \eqref{eq2.21}, if
$$
(f_{{\x}^*,\mu}^R)'(\y^*;\y-\y^*)\leq0,\quad\forall \y\in\Y,$$
which is equivalent to that
for any $p_{j}^{*}\in\mathcal{D}(h_{j}(\y^*_j)-\mu)$ and $p_{m+j}^{*}\in\mathcal{D}(h_{m+j}(\y^*_j)-\mu)$, it holds
\begin{equation}\label{eq2.25}
\begin{aligned}
&\frac{\lambda_2}{\mu}\sum_{j=1}^m\varpi_{p_{j}^{*}}'(t)_{t=h_{j}(\y^*_j)-\mu}
h_{j}'(\y^*_j)\e_j+\frac{\lambda_2}{\mu}\sum_{j=1}^m\varpi_{p_{m+j}^{*}}'(t)_{t=h_{m+j}(\y^*_j)-\mu}
h_{m+j}'(\y^*_j)\e_j\\
\in&-\partial_{\y} c(\x^*,\y^*)+
\frac{\lambda_2}{\mu}\partial\left(\sum_{j=1}^m\varpi(h_{j}(\y^*_j))+\varpi(h_{m+j}(\y^*_j))\right)
+\mathcal{N}_{\Y}(\y^*).
\end{aligned}\end{equation}

Based on the above analysis, we introduce the following definitions to min-max problem \eqref{obb-3r}.
\begin{definition}\label{def-d}
 For $(\x^*,\y^*)\in\X\times\Y$,
\begin{itemize}
\item if \eqref{eq2.25-2} and \eqref{eq2.25} hold for all
$q_{i}^{*}\in\mathcal{D}(g_{i}(\x^*_i)-\mu)$,
$q_{n+i}^{*}\in\mathcal{D}(g_{n+i}(\x^*_i)-\mu)$, $p_{j}^{*}\in\mathcal{D}(h_{j}(\y^*_j)-\mu)$ and $p_{m+j}^{*}\in\mathcal{D}(h_{m+j}(\y^*_j)-\mu)$,
we call $(\x^*,\y^*)$
 a \textbf{{d-stationary point}} of min-max problem \eqref{obb-3r};
\item
if there exist $q_{i}^{*}\in\mathcal{D}(g_{i}(\x^*_i)-\mu)$,
$q_{n+i}^{*}\in\mathcal{D}(g_{n+i}(\x^*_i)-\mu)$, $p_{j}^{*}\in\mathcal{D}(h_{j}(\y^*_j)-\mu)$ and $p_{m+j}^{*}\in\mathcal{D}(h_{m+j}(\y^*_j)-\mu)$ such that \eqref{eq2.25-2} and \eqref{eq2.25} hold, we call $(\x^*,\y^*)$ a \textbf{weak d-stationary point} of min-max problem \eqref{obb-3r}.
\end{itemize}
\end{definition}

On the one hand, if $(\x^*,\y^*)$ is a local saddle point of problem \eqref{obb-3r}, then it is  a (weak) d-stationary point of \eqref{obb-3r}. On the other hand, if $(\x^*,\y^*)\in\X\times\Y$ is a weak d-stationary point of \eqref{obb-3r}, then it satisfies \eqref{first1} and \eqref{first2}.
\begin{proposition}\label{prop6}
Let density function $\rho$ satisfy Assumption \ref{ass7} and $0<\mu<\bar{\mu}_1$ with $\bar{\mu}_1$ defined in \eqref{eq-mu}. If $(\x^*,\y^*)$ is a weak d-stationary point of \eqref{obb-3r}, then the following statements hold.
\begin{itemize}
\item [{\rm (i)}] $g_{\hat{i}}({\x}^*_i)\not\in(0,\mu)$, $\forall i\in[{n}]$, $\hat{i}\in\{i,n+i\}$ and
$h_{\hat{j}}({\y}^*_j)\not\in(0,\mu)$,
$\forall j\in[{m}]$, $\hat{j}\in\{j,m+j\}$;
\item [{\rm (ii)}] for $\hat{i}\in\{i, n+i\}$, if $g_{\hat{i}}(\x^*_i)=\mu$ and
$q_{\hat{i}}^{*}\in\mathcal{D}(g_{\hat{i}}(\x^*_i)-\mu)$
such that \eqref{eq2.25-2} holds, then $q_{\hat{i}}^{*}=1$;
\item [{\rm (iii)}] for $\hat{j}\in\{j,m+j\}$, if $h_{\hat{j}}(\y^*_j)=\mu$ and
$p_{\hat{j}}^{*}\in\mathcal{D}(h_{\hat{j}}(\y^*_j)-\mu)$
such that \eqref{eq2.25} holds, then $p_{\hat{j}}^{*}=1$.
\end{itemize}
\end{proposition}
\begin{proof}
From Proposition \ref{prop2}, (i) holds naturally. Next, we argue
(ii) by contradiction and (iii) can be proved similarly.
For item (ii), suppose there exists ${\tilde{i}}\in[{n}]$ such that $g_{\tilde{i}}(\x^*_{\tilde{i}})=\mu$ and \eqref{eq2.25-2} holds with $q_{\tilde{i}}^{*}=2$. By Assumption \ref{ass4} and $\mu<\bar{\mu}_1\leq\tau$, $\x_{\tilde{i}}^*\in{\rm int}(\tilde{\X}_{\tilde{i}})$ and $g_{n+\tilde{i}}(\x^*_{\tilde{i}})\not\in[0,\tau]$, which implies $q_{n+\tilde{i}}^{*}\in\mathcal{D}(g_{n+\tilde{i}}(\x^*_{\tilde{i}})-\mu)$ is unique and $\varpi_{q_{n+\tilde{i}}^{*}}'(t)_{t=g_{n+\tilde{i}}(\x^*_{\tilde{i}})-\mu}
=\varpi'(t)_{t=g_{n+\tilde{i}}(\x^*_{\tilde{i}})}$. Then, \eqref{eq2.25-2} gives
\begin{equation}\label{eq2.4}
0\in[\partial_{\x} c(\x^*,\y^*)]_{\tilde{i}}+\frac{\lambda_1}{\mu}g_{\tilde{i}}'(\x^*_{\tilde{i}})
+[\mathcal{N}_{{\X}}(\x^*)]_{\tilde{i}}.
\end{equation}

Using Assumption \ref{ass4}, \eqref{eq3-3} and \eqref{eq3-1}, we confirm that
$\frac{\lambda_1\sigma}{\mu}\leq L_{c,1}$,
which contradicts to the supposition on the value of $\mu$ and gives the result in (ii).
\end{proof}

For $t\in\R$, denote $\A_{i}^0(t)=\{\hat{i}\in\{i,n+i\}:g_{\hat{i}}(t)=0\}$ and $\mathcal{B}_{j}^0(t)=\{\hat{j}\in\{j,m+j\}:h_{\hat{j}}(t)=0\}$.
For given $(\x^*,\y^*)\in\X\times\Y$ and $\mu\in\R_{++}$, consider the following two functions
\begin{equation}\label{eq-W2}
W_{\x^*,\y^*,\mu}(\x)=c(\x,\y^*)+
\frac{\lambda_1}{\mu}\sum\nolimits_{i=1}^n\sum\nolimits_{\hat{i}\in\A_i^0(\x_i^*)}g_{\hat{i}}(\x_i)_+,
\end{equation}
\begin{equation}\label{eq-V}
V_{\x^*,\y^*,\mu}(\y)=-c(\x^*,\y)+\frac{\lambda_2}{\mu}\sum\nolimits_{j=1}^m
\sum\nolimits_{\hat{j}\in\mathcal{B}_j^0(\y_j^*)}h_{\hat{j}}(\y_j)_+.
\end{equation}

If $g_i$ is convex for any $i\in[2n]$, by the convexity of $c(\cdot,\y^*)$, function $W_{\x^*,\y^*,\mu}$ is convex on $\X$, which gives that ${\x}^*$ is a minimizer of $W_{\x^*,\y^*,\mu}$ on $\X$ if and only if
\begin{equation}\label{eq-W}
\begin{aligned}
0\in&\partial W_{\x^*,\y^*,\mu}({\x}^*)+\mathcal{N}_{\X}({\x}^*)\\
=&\partial_{\x} c({\x}^*,\y^*)+
\frac{\lambda_1}{\mu}\sum\nolimits_{i=1}^n\left(\sum\nolimits_{\hat{i}\in\A_i^0(\x_i^*)}[0,1]
g_{\hat{i}}'({{\x}}_i^*)\right)\e_i+\mathcal{N}_{\X}({\x}^*).
\end{aligned}\end{equation}
Similarly, if $h_j$ is convex for any $j\in[2m]$, $V_{\x^*,\y^*,\mu}$ is convex on $\Y$ and ${\y}^*\in\Y$ is a minimizer of $V_{\x^*,\y^*,\mu}$ on $\Y$ if and only if
$$                       0\in\partial V_{\x^*,\y^*,\mu}({\y}^*)+\mathcal{N}_{\Y}({\y}^*).
$$

In what follows, we will verify that all weak d-stationary points of \eqref{obb-3r} are the $\mu$-strong local saddle points of \eqref{obb-3-1}.
\begin{theorem}\label{theorem4}
Under conditions of Proposition \ref{prop6},
if $(\x^*,\y^*)\in\X\times\Y$ is a weak d-stationary point of \eqref{obb-3r},
then it is a $\mu$-strong local saddle point of problem \eqref{obb-3-1}.
\end{theorem}
\begin{proof}
Since $(\x^*,\y^*)$ is a weak d-stationary point of \eqref{obb-3r}, putting forward the results in Proposition \ref{prop6} to \eqref{eq2.25-2}, we have \eqref{eq-W}, which means that $\x^*$ is a global minimizer of $W_{\x^*,\y^*,\mu}$ on $\X$, i.e.
\begin{equation}\label{eq-l1}
c(\x^*,{\y^*})
\leq c(\x,{\y^*})+
\frac{\lambda_1}{\mu}\sum\nolimits_{i=1}^n\sum\nolimits_{\hat{i}\in\A_i^0(\x_i^*)}g_{\hat{i}}(\x_i)_+,\quad\forall\, \x\in\X.\end{equation}
Then, \eqref{eq-l1} means that for any $\x\in\{\x\in\X:g_{\hat{i}}(\x_i)\leq0\,\, \mbox{if}\,\, g_{\hat{i}}(\x_i^*)\leq0,\,\,\forall i\in[n],\,\hat{i}\in\{i,n+i\}\}$, it holds
$c(\x^*,{\y^*})
\leq c(\x,{\y^*})$.
Similarly, Proposition \ref{prop6} together with \eqref{eq-V} implies that $\y^*$ is a global minimizer of $V_{\x^*,\y^*,\mu}$ on $\Y$ and we further have that
$\y^*$ is a maximizer of $c({\x^*},\cdot)$ on $\{\y\in\Y:h_{\hat{j}}(\y)\leq0\,\, \mbox{if}\,\, h_{\hat{j}}(\y_j^*)\leq0,\,\forall j\in[m],\,\hat{j}\in\{j,m+j\}\}$. Thus, from Theorem \ref{theorem2} and recalling Proposition \ref{prop6}-(i),
$(\x^*,\y^*)$ is a $\mu$-strong local saddle point of \eqref{obb-3-1}.
\end{proof}
\begin{remark}
Following the proof of Proposition \ref{prop6}, when $0<\mu<\bar{\mu}_1$, if $\x^*$ and $\y^*$ satisfy
$$\x^*\in \arg\min_{\x\in\X}W_{\x^*,\y^*,\mu}(\x)\quad\mbox{and}\quad \y^*\in\arg\min_{\y\in\Y}V_{\x^*,\y^*,\mu}(\y),$$
 then $(\x^*,\y^*)$ is a $\mu$-strong local saddle point of \eqref{obb-3-1}.
\end{remark}

By \cite[Proposition 17]{Jin-Netrapalli-Jordan}, any local saddle point is a local minimax point. Then, by Theorem \ref{theorem4}, any weak d-stationary point of problem \eqref{obb-3r} is
also a local minimax point of problem \eqref{obb-3-1}.

Since the continuous relaxation functions to the cardinality functions in \eqref{obb-3r} are DC functions and variable separated, the proximal operator of its subtracted convex function can be calculated directly in most cases. Moreover, to solve problem \eqref{obb-3r} with a nonsmooth function $c$ efficiently, we can use a
smoothing approximation of \eqref{obb-3r} as follows
\begin{equation}\label{obb-3rS}
\min_{\x\in\X}\max_{\y\in\Y} \,f^R(\x,\y,\mu,\varepsilon):=\tilde{c}(\x,\y, \varepsilon)+\lambda_1\sum_{i=1}^n {\phi}_{i}^{R}(\x_i,\mu)
-\lambda_2\sum_{j=1}^m{\psi}_{j}^{R}(\y_j,\mu),
\end{equation}
where $\tilde{c}$ is a  smoothing function of $c$ defined by \eqref{smoothing}.
Similar to the expression in \eqref{eqfR} and by Proposition \ref{prop-c}-(ii), for fixed $\mu>0$ and $\varepsilon>0$, $\tilde{f}^R(\x,\y,\mu,\varepsilon)$ in \eqref{ob-cr} is a DC function with respect to $\x$ and $\y$, respectively. Thus, the d-stationary point and weak d-stationary point to \eqref{obb-3rS} can be defined according to Definition \ref{def-d}. By using the gradient consistency \eqref{s2}-\eqref{s3}, we have the following result.
\begin{proposition}\label{prop10}
Let $\tilde{c}:\R^n\times\R^m\times(0,1]\rightarrow\mathbb{R}$ be defined as in \eqref{smoothing}.
If $\{(\x^k,\y^k)\}$ is a sequence of weak d-stationary points of \eqref{obb-3rS} with $\varepsilon_k\downarrow0$, then any accumulation point of $\{(\x^k,\y^k)\}$ is a weak d-stationary point of \eqref{obb-3r}.
\end{proposition}

\subsection{Density function $\rho$ under Assumption  \ref{ass8}}\label{section5.2}
Section \ref{section4} focuses on the study of \eqref{obb-3-1} with density function $\rho$ satisfying Assumption \ref{ass7}. From Table \ref{table1}, we find that three density functions satisfy Assumption \ref{ass8} and the corresponding continuous relaxation function $r(\cdot,\mu)$ owns the continuous differentiability on $\R_{++}$, which may bring some convenience to its algorithm research when $c$ is smooth.
Thus,
in this subsection, we pay attention to the results of the continuous relaxation to \eqref{obb-3-1} with density function $\rho$ satisfying Assumption \ref{ass8} and consider \eqref{obb-3-1} under the following conditions:
\begin{itemize}
\item [{\rm (i)}] functions $c$, $g_i$, $i\in[2n]$ and $h_j$, $j\in[2m]$ are Lipschitz continuously differentiable;
\item [{\rm (ii)}]  the feasible regions are defined by the box constraints, i.e.
$$
{\X}=\{{\x}\in \R^n:\underline{\u}\leq \x\leq \overline{{\u}}\},\quad {{\Y}}=\{{\y}\in \R^m:\underline{\v}\leq \y\leq \overline{\v}\},$$
where $\underline{\u}$, $\overline{{\u}}$, $\underline{\v}$ and $\overline{\v}$ are defined as in \eqref{eq-cons};
\item [{\rm (iii)}]  Assumption \ref{ass4} holds, in which the conditions of \eqref{ass2-3} and \eqref{ass2-4} can be ignored.
\end{itemize}
In this case, \eqref{ass-X} holds naturally and we will
consider the second order necessary optimality condition of \eqref{obb-3-1}.

To proceed, we first introduce some notations on the existing parameters.
\begin{itemize}
\item By virtue of the Lipschitz continuous differentiability of $c$ on $\X\times\Y$, there exists 
$L_{c,2}$ such that for any $\x\in\X$, $\y\in\Y$, it holds
{$$\sup\{|H_{ii}|,|M_{jj}|:H\in\partial_{\x\x}^2c(\x,\y),M\in\partial_{\y\y}^2c(\x,\y),i\in[n], j\in[m]\}\leq L_{c,2}.$$}
\item Since $g_{{i}}:\R\rightarrow\R$ is Lipschitz continuous differentiable on $\X_i$, and $\X_i$ is compact, there exists $L_{g,2}$ such that
    $$\sup\{|\xi|:\xi\in\partial^2g_{\tilde{{i}}}(\x_i),\x_i\in\X_i,i\in[n],\tilde{i}={i,n+i}\}\leq L_{g,2}.$$
    Similarly, there exists $L_{h,2}$ such that  $$\sup\{|\eta|:\eta\in\partial^2h_{\tilde{{j}}}(\y_j),\y_j\in\Y_j,j\in[m],\tilde{j}={j,m+j}\}\leq L_{h,2}.$$
\item {Under Assumption \ref{ass8}, there exist
$\bar{\rho}_2>0$ and $\rho_0>0$ such that}
\begin{equation}\label{eq-rho1}
\rho(s)\leq \bar{\rho}_2,\quad\forall s\in(0,\alpha)\quad\mbox{and}\quad \lim\nolimits_{t\downarrow0}\rho(t)=\rho_0.
\end{equation}
\item For $\x\in\X$, $\y\in\Y$ and $\delta>0$, denote
$$\mathcal{A}^+_{\delta}(\x)=\{i\in[n]:\,\mbox{$0<g_{{i}}({\x}_{{i}})<\delta$ or $0<g_{n+{i}}({\x}_{{i}})<\delta$}\},$$
$$\mathcal{B}^+_{\delta}(\y)=\{j\in[m]:\,\mbox{$0<h_{{j}}({\y}_{j})<\delta$ or $0<h_{m+{j}}({\y}_{j})<\delta$}\}.$$
\end{itemize}

Proceed to the next step, and let
\begin{equation}\label{eq-mu2}
\bar{\mu}_2=\min\left\{\frac{\tau}{\alpha},\frac{{\lambda}\sigma\underline{\rho}_2}{L_{c,1}},
\frac{{\lambda_1}\check{\rho}_2\sigma^2}{\tau L_{c,2}/\alpha+\lambda_1\bar{\rho}_2L_{g,2}},
\frac{{\lambda_2}\check{\rho}_2\sigma^2}{\tau L_{c,2}/\alpha+\lambda_2\bar{\rho}_2L_{h,2}}
\right\}
\end{equation}
with ${\lambda}=\min\{\lambda_1,\lambda_2\}$. In particular, when $g_{i}$, $h_{j}$ are linear functions and $c(\cdot,\cdot)$ is also linear with respect to $\x$ and $\y$, respectively, then $\bar{\mu}_2=\left\{\frac{\tau}{\alpha},\frac{{\lambda}\sigma\underline{\rho}_2}{L_{c,1}}\right\}$. If we further
choose $\rho$ as in Example \ref{example4.3} or Example \ref{example4.4}, then,
$\bar{\mu}_2=\frac{\tau}{\alpha}.$

In what follows, suppose that $\rho$ satisfies Assumption \ref{ass8} and $0<\mu<\bar{\mu}_2$ with $\bar{\mu}_2$ defined in \eqref{eq-mu2}.
By the Lipschitz continuity of $\rho$ on $\R_{++}$, $r(g_{i}(t),\mu)$ is Lipschitz continuous differentiable on $\{t:g_{i}(t)>0\}$, $\forall i\in[2n]$.
When $g_{i}(t)>0$, by \eqref{sec-r}, the second order generalized derivative of $r(g_{i}(t),\mu)$ with respect to $t$ satisfies
\begin{equation}\label{eq2.26}
\partial_{tt}^2r(g_{i}(t),\mu)\subseteq \tilde{\partial}_{tt}^2r(g_{i}(t),\mu)
:=\frac{\partial\rho(s)_{s=g_{i}(t)/\mu}}{\mu^2}(g_{i}'(t))^2+
\frac{\rho(g_{i}(t)/\mu)}{\mu}\partial^2g_{i}(t),
\end{equation}
where $\rho(g_{i}(t)/\mu)=0$ if $g_{i}(t)\not\in[0,\tau)$. Then, $\tilde{\partial}_{tt}^2r(g_{i}(t),\mu)=\{0\}$ when $g_{i}(t)\not\in[0,\tau)$.
Thus, if $\x^*$ is a local solution of $\min_{\x\in\X}\,f^R(\x,{\y}^*,\mu)$ and $\mathcal{A}^+_{\tau}(\x^*)\neq\emptyset$, by Assumption \ref{ass4} and the second order necessary optimality condition \cite{Jean1984Generalized}, for any ${i}\in\mathcal{A}^+_{\tau}(\x^*)$,
\begin{equation}\label{eq2.27}
\left\{\begin{aligned}
&\mbox{there exists a unique $\hat{i}\in\{i,n+i\}$ such that $0<g_{\hat{i}}({\x}_{{i}}^*)<\tau$ and}\\
&\mbox{there exists a $\omega_i\in[\partial^2_{\x\x}c(\x^*,{\y}^*)]_{ii}+
\lambda_1\tilde{\partial}^2_{tt}r(g_{\hat{i}}(t),\mu)_{t={\x}_{{i}}^*}$ such that $\omega_i\geq0$,}
\end{aligned}\right.\end{equation}
which implies $\x_{{i}}^*\in{\rm int}({\X}_{{i}})$. Similarly,
if $\y^*$ is a local solution of $\max_{\y\in\Y}\,f^R(\x^*,{\y},\mu)$ and $\mathcal{B}^+_{\tau}(\y^*)\neq\emptyset$, for any ${j}\in\mathcal{B}^+_{\tau}(\y^*)$,
\begin{equation}\label{eq2.27-2}
\left\{\begin{aligned}
&\mbox{ there exists a unique $\hat{j}\in\{j,m+j\}$ such that $0<h_{\hat{j}}({\y}_{{j}}^*)<\tau$ and}\\
&\mbox{ there exists a $\varpi_j\in-[\partial^2_{\y\y}c(\x^*,{\y}^*)]_{jj}+
\lambda_2\tilde{\partial}^2_{tt}r(h_{\hat{j}}(t),\mu)_{t={\y}_{j}^*}$ such that $\varpi_j\geq0$.}
\end{aligned}\right.\end{equation}

Thus, inspired by the first and second order necessary conditions to
\begin{equation}
\x^*\in\arg\min_{\x\in\mathcal{X}}f^R(\x,\y^*,\mu)\quad\mbox{and}
\quad\y^*\in\arg\max_{\y\in\mathcal{Y}}f^R(\x^*,\y,\mu),
\end{equation}
we introduce the following definition.
\begin{definition}\label{def-s}
We call $(\x^*,\y^*)\in\X\times\Y$ a {\bf{weak second order stationary point}} of problem \eqref{obb-3r}, if
\begin{equation}\label{first1-3}
\0\in\tilde{{\partial}}_{\x}f^R(\x^*,{\y}^*,\mu)+\mathcal{N}_{\X}(\x^*)\quad\mbox{and}\quad
\0\in-\tilde{{\partial}}_{\y}f^R({\x}^*,\y^*,\mu)+\mathcal{N}_{\Y}(\y^*),
\end{equation}
where $\tilde{\partial}_{\x}f^R(\x^*,{\y}^*,\mu)$ and $\tilde{{\partial}}_{\y}f^R({\x}^*,\y^*,\mu)$ are defined  in \eqref{first1-2} and \eqref{first2-2},
and for any
${i}\in\mathcal{A}^+_{\tau}(\x^*)$, $j\in \mathcal{B}^+_{\tau}(\y^*)$, \eqref{eq2.27} and \eqref{eq2.27-2} hold, respectively.
\end{definition}

It is clear that \eqref{first1-3} and \eqref{eq2.27}-\eqref{eq2.27-2} are weaker than the general first and second order necessary optimality conditions to \eqref{obb-3r}, respectively, so we call it ``weak" stationary point.
\begin{theorem}\label{theorem6}
Suppose problem \eqref{obb-3-1} satisfies Assumption \ref{ass4},
density function $\rho$ satisfies Assumption \ref{ass8} and $0<\mu<\bar{\mu}_2$ with $\bar{\mu}_2$ defined in \eqref{eq-mu2}. Then, the following statements hold.
\begin{itemize}
\item [{\rm (i)}] If $(\x^*,\y^*)$ is a weak second order stationary point of \eqref{obb-3r}, then
\begin{equation}\label{lower2}
g_{\hat{{i}}}(\x^*_i)\not\in(0,\alpha\mu),\,\forall i\in[{n}],\,\hat{i}\in\{i,n+i\};\;
h_{\hat{{j}}}(\y^*_j)\not\in(0,\alpha\mu),\, \forall j\in[{m}],\,\hat{j}\in\{j,m+j\}.
\end{equation}
\item [{\rm (ii)}]
$(\x^*,\y^*)$ is a saddle point of problem \eqref{obb-3-1} if and only if it is a
saddle point of \eqref{obb-3r}.
\item [{\rm (iii)}] $(\x^*,\y^*)$ is an
$\alpha\mu$-strong local saddle point of \eqref{obb-3-1} if it is a local saddle point of \eqref{obb-3r}.
\item [{\rm (iv)}] When functions $g_{i}$, $h_{j}$ are convex for all $i\in[2{n}]$ and $j\in[2{m}]$, $(\x^*,\y^*)$ is an $\alpha\mu$-strong local saddle point of \eqref{obb-3-1} if it is a weak second order stationary point of \eqref{obb-3r}.
\end{itemize}
\end{theorem}
\begin{proof}
To prove (i), we
argue the results in \eqref{lower2} by contradiction. Suppose
there exists $\tilde{i}\in[n]$ such that $0<g_{\tilde{i}}(\x^*_{\tilde{i}})<\alpha\mu$.

By Assumption \ref{ass4}, since $\alpha\mu<\tau$, then
$\x_{\tilde{i}}^*\in{\rm int}({\X}_{\tilde{i}})$ and $g_{n+\tilde{i}}(\x_{\tilde{i}}^*)\not\in[0,\tau]$, which together with \eqref{eq-rho} implies that
$$
\nabla_t r(t,\mu)_{t=g_{n+\tilde{i}}(\x_{\tilde{i}}^*)}=0\quad\mbox{and}\quad
\nabla_t^2r(t,\mu)_{t=g_{n+\tilde{i}}(\x_{\tilde{i}}^*)}=0.$$

 Next, we obtain the contradiction to $0<g_{\tilde{i}}(\x^*_{\tilde{i}})<\alpha\mu$ from two cases.

 Case 1: $\sup\{a:a\in\partial\rho(g_{\tilde{i}}(\x_{\tilde{i}}^*)/\mu)\}>-\check{\rho}_2$. By Assumption \ref{ass8}, it means that $\rho(g_{\tilde{i}}(\x_{\tilde{i}}^*)/\mu)\geq \underline{\rho}_2$. Similar to the discussion in Proposition \ref{prop2}, by $\mu<\bar{\mu}_2\leq\lambda_1\sigma\underline{\rho_2}/L_{c,1}$, it brings a contradiction.

Case 2: $\sup\{a:a\in\partial\rho(g_{\tilde{i}}(\x_{\tilde{i}}^*)/\mu)\}\leq-\check{\rho}_2$. By \eqref{eq2.26} and \eqref{eq2.27},
there exist $\xi_{\tilde{i}}\in[\partial^2_{\x\x}c(\x^*,{\y}^*)]_{{\tilde{i}}{\tilde{i}}}$, $\eta_{\tilde{i}}\in \partial\rho(t)_{t=g_{\tilde{i}}({\x}_{\tilde{i}}^*)/\mu}$ and $\zeta_{\tilde{i}}\in\partial^2g_{\tilde{i}}(t)_{t={\x}_{\tilde{i}^*}}$ such that
\begin{equation}\label{eq2.28}
\begin{aligned}
\xi_{\tilde{i}}+\lambda_1
\frac{\eta_{\tilde{i}}}{\mu^2}(g_{\tilde{i}}'({\x}_{\tilde{i}}^*))^2+\lambda_1
\frac{\rho(g_{\tilde{i}}({\x}_{\tilde{i}}^*)/\mu)}{\mu}\zeta_{\tilde{i}}\geq0
\end{aligned}.
\end{equation}

Recalling Assumption \ref{ass8} and by $\mu<\frac{\tau}{\alpha}$, an estimation on the left side of \eqref{eq2.28} gives
\begin{equation}\label{eq2.29}
0\leq \frac{L_{c,2}{\tau}}{\alpha}-\lambda_1\check{\rho}_2\sigma^2/\mu+\lambda_1\bar{\rho}_2L_{g,2},
\end{equation}
which contradicts to $\mu<\bar{\mu}_2\leq\frac{\lambda_1\check{\rho}_2\sigma^2}{\tau L_{c,2}/\alpha+\lambda_1\bar{\rho}_2L_{g,2}}$ given in \eqref{eq-mu2}. Therefore,
$g_{{{i}}}(\x^*_{{i}})\not\in(0,\alpha\mu)$, $\forall i\in[n]$.
Similarly, $g_{n+i}(\x^*_{{i}})\not\in(0,\alpha\mu)$, $\forall i\in[n]$, and $h_{{j}}(\y^*_j)\not\in(0,\alpha\mu)$, $h_{m+{j}}(\y^*_j)\not\in(0,\alpha\mu)$, $\forall j\in[m]$. Thus, (i) holds, which further implies Assumption \ref{ass6}. By Theorem \ref{theorem1}, we can obtain the results in (ii) and (iii).

(iv) Suppose $(\x^*,\y^*)$ is a weak second order stationary point of \eqref{obb-3r}. To proceed the proof, we use a slight modification of functions in \eqref{eq-W2} and \eqref{eq-V} as follows
\begin{equation}\label{eq-W22}
W_{\x^*,\y^*,\mu}(\x)=c(\x,\y^*)+
\frac{\lambda_1\rho_0}{\mu}\sum\nolimits_{i=1}^n\sum\nolimits_{\hat{i}\in\A_i^0(\x_i^*)}g_{\hat{i}}(\x_i)_+
,
\end{equation}
\begin{equation}\label{eq-V2}
V_{\x^*,\y^*,\mu}(\y)=-c(\x^*,\y)+\frac{\lambda_2\rho_0}{\mu}
\sum\nolimits_{j=1}^m\sum\nolimits_{\hat{j}\in\mathcal{B}_j^0(\y_j^*)}h_{\hat{j}}(\y_j)_+,
\end{equation}
which are convex on $\X$ and $\Y$, respectively.
By Assumption \ref{ass8} and \eqref{eq-19}, if $g_{\hat{i}}(\x_i^*)\geq\alpha\mu$ or $g_{\hat{i}}(\x_i^*)<0$ for some $i\in[n]$ and $\hat{i}\in\{i,n+i\}$, then
$\partial_{t} r(t,\mu)_{t=g_{\hat{i}}(\x_i)}g_{\hat{i}}'(\x_i)=0$.
From (i), $(\x^*,\y^*)$ satisfies \eqref{lower2}. Thus,
using \eqref{eq-19} again and by \eqref{first1-3}, we have
\begin{equation}\label{eq2.13}
\begin{aligned}
0\in\nabla_{\x}c(\x^*,{\y}^*)
+
\frac{\lambda_1\rho_0}{\mu}\sum\nolimits_{i=1}^n\left(\sum\nolimits_{\hat{i}\in\A_i^0(\x_i^*)}
[0,1]g_{\hat{i}}'({{\x}}_i^*)\right)\e_i+\mathcal{N}_{\X}(\x^*),
\end{aligned}\end{equation}
which implies
$0\in\partial_{\x}{W}_{\x^*,\y^*,\mu}(\x^*)+\mathcal{N}_{\X}(\x^*)$. Thus, $\x^*$ is a global minimizer of ${W}_{\x^*,\y^*,\mu}(\x)$ on $\X$. In what follows, similar to the analysis in Theorem \ref{theorem4}, we get that $(\x^*,\y^*)$ is an $\alpha\mu$-strong local saddle point of \eqref{obb-3-1}.
\end{proof}
\subsection{Continuous relaxations defined by different density functions}
 In this subsection, we use three examples to explain the different properties of the continuous relaxation problems constructed by the density functions that satisfy Assumption \ref{ass7} or \ref{ass8}. In particular, we use the density functions in Examples \ref{example4.1} and  \ref{example4.3} to construct two different continuous relaxation problems, which have different relations with min-max problem \eqref{obb} regarding local saddle points and strong local saddle points.
\begin{itemize}
\item In Example \ref{example5.1}, we show that we can provide a possible larger lower bound to the saddle points of \eqref{obb-3-1} by the analysis on the continuous relaxation models with different density functions.
\item
It is interesting to see in Example \ref{example5.2} that the bounds in
\eqref{eq-ass6-1} and \eqref{eq-ass6-2}
with $0<\mu<\bar{\mu}_1$, $\bar{\mu}_1$ in \eqref{eq-mu} and $\alpha=1$ (given in subsection \ref{section5.1} by the continuous relation model with density function in Example \ref{example4.1})
     is satisfied by the global minimax points of this example, but these bounds with $0<\mu<\bar{\mu}_2$, $\bar{\mu}_2$ in \eqref{eq-mu2}, and $\alpha$ in Assumption \ref{ass8} (given in subsection \ref{section5.1} by the continuous relation model with a density function satisfying Assumption \ref{ass8}) may not hold to the global minimax points.
\item Note that all the functions $r(\cdot,\mu)$ in Examples \ref{example4.1}-\ref{example4.4} can be expressed by DC functions and continuously differentiable on $(0,\alpha\mu)$. Then, when $c$ is continuously differentiable, both the weak d-stationary point and
weak second order stationary point to the continuous model with density functions in Examples \ref{example4.1}-\ref{example4.4} are well-defined.
  In Example \ref{example5.3}, we will show that
 a weak second order stationary point is not necessary to be a weak d-stationary {point} of the continuous relaxation problem with density function in Example \ref{example4.1}. Moreover, a weak d-stationary points is also not necessary to be a weak second order stationary point of the continuous relaxation problem with density function in Example 3.3.
\end{itemize}
\begin{example}\label{example5.1}
Consider
\begin{equation}\label{eq-ex1}
\min_{\x\in[-2,2]}\max_{\y\in[-2,2]}f(\x,\y):=(\x-1)(\y-1)+3\|\x\|_0-3\|\y\|_0.
\end{equation}
Example \ref{example-s} has verified that $(0,0)$ is the unique saddle point of \eqref{eq-ex1}. Assumption \ref{ass4} holds with $\sigma=1$, $\tau=2$, $L_{c,1}=3$ and $L_{c,2}=L_{g,2}=L_{h,2}=0$.

Case 1: Choose the density function $\rho$ in Example \ref{example4.1} to build up its continuous relaxation. Then, $\alpha=1$, $\underline{\rho}=1$ and then $\bar{\mu}_1=1$ in \eqref{eq-mu}. Since we can choose any $\mu$ in $(0,\bar{\mu}_1)$, by Theorem \ref{theorem5}, it gives that the saddle points and global minimax points of \eqref{eq-ex1} satisfy the lower bounds that
\begin{equation}\label{eq-ex5-1}
\mbox{either $\x=0$ or $|\x|\geq\nu$  and either $\y=0$ or $|\y|\geq\nu$}
\end{equation}
with $\nu=1$.

Case 2: Choose the density function $\rho$ in Example \ref{example4.3} with $\alpha=2$ to build up its continuous relaxation. Then, the analysis in subsection \ref{section5.2} gives that $\bar{\mu}_2=1$.  By Theorem \ref{theorem6}, we have that
any saddle point of \eqref{eq-ex1} satisfies the lower bound in \eqref{eq-ex5-1} with $\nu=2$.
\end{example}
\begin{example}\label{example5.2}
Consider
$$
\min_{\x\in[-2,2]}\max_{\y\in[-2,2]}f(\x,\y):=(\x-1)(\y-1)+\|\x\|_0-\|\y\|_0.$$
Example \ref{example-s} shows that $({3}/{2},0)$ and $({3}/{2},2)$ are global minimax points of this problem.

By basic calculation, $\bar{\mu}_1=1/3$ when we define $\rho$ by Example \ref{example4.1}. Then, by Theorem \ref{theorem5},
any global minimax point of this example satisfies \eqref{eq-ex5-1} with any $\nu=\mu<\bar{\mu}_1$.

However, when we define $\rho$ by Example \ref{example4.3} with $\alpha=2$, then $\bar{\mu}_2=1$.
It is obvious that neither of the two global minimax points satisfies \eqref{eq-ex5-1} with
$\nu=\alpha\mu$ when $3/4<\mu<\bar{\mu}_2$.
\end{example}
\begin{example}\label{example5.3}
Consider
\begin{equation}\label{eq-ex5-2}
\min_{\x\in[-2,2]}\max_{\y\in[-2,2]}f(\x,\y):=(\x-1)(1-\y)+\|\x\|_0-\|\y\|_0.
\end{equation}

On the one hand, choose the density function $\rho$ in Example \ref{example4.1} with $\mu=1/4$ to build up its continuous relaxation, where $0<\mu<\bar{\mu}_1=1/3$. For this case, we can verify that $(-1/4,1/4)$ is a weak second order stationary point of its continuous relaxation model, but it is not a weak d-stationary point of it and is also not a local saddle point of \eqref{eq-ex5-2}.

On the other hand, choose the density function $\rho$ in Example \ref{example4.3} with $\alpha=1$ and $\mu=1$ to build up its continuous relaxation, where $0<\mu<\bar{\mu}_2=2$.
For this case, we can easily check that $(1,1)$ is a weak d-stationary point but not a weak second order stationary point of this continuous relaxation model, and not a local saddle point of \eqref{eq-ex5-2}.
\end{example}

At the end of this subsection, we summarize the relations between min-max problem \eqref{obb-3-1} and its continuous relaxation problem \eqref{obb-3r} in Fig. \ref{fig1}.
\begin{figure}
\centering
\includegraphics[width=5.2in,height=2.5in]{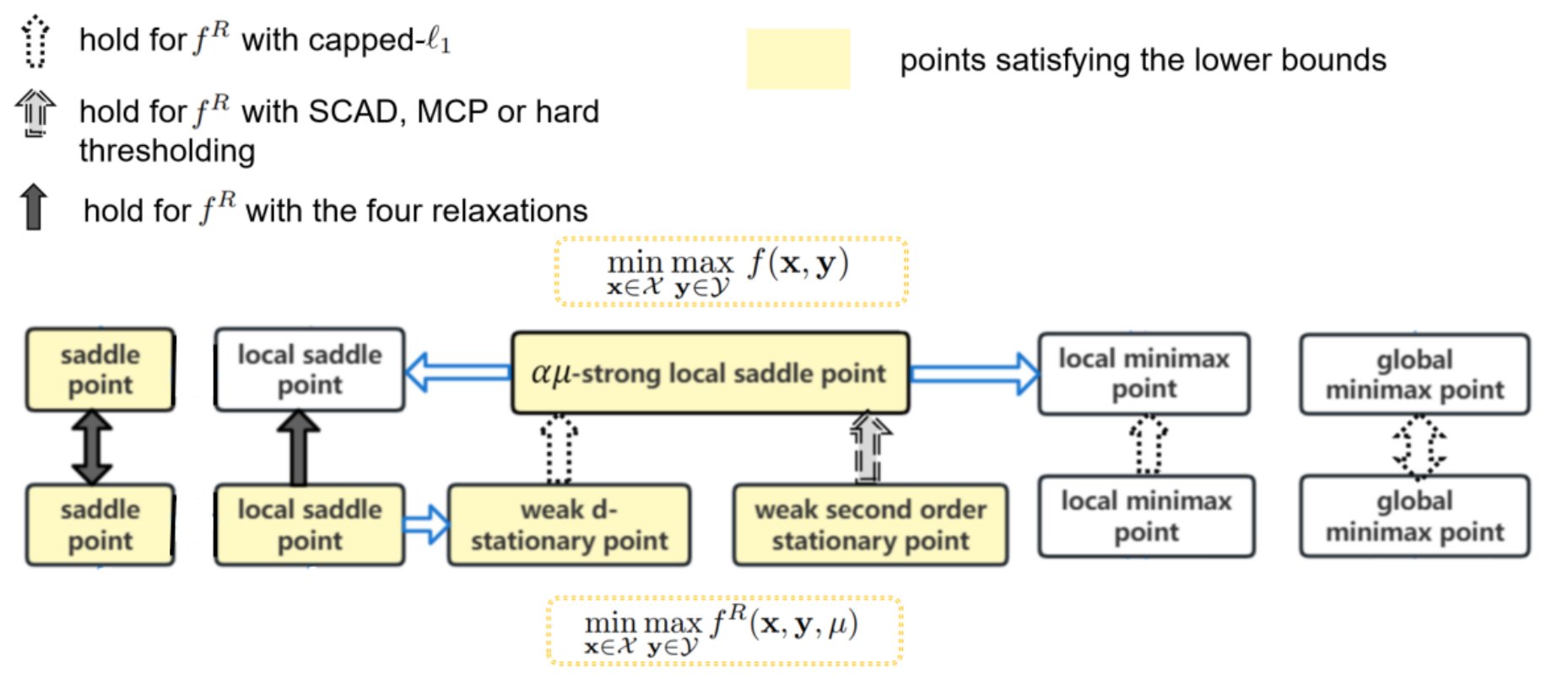}
\caption{Relations between problems \eqref{obb-3-1} and \eqref{obb-3r} with different relaxations}\label{fig1}
\end{figure}

\section{Applications}\label{section6}
In this section, we use three examples to explain the motivation and theoretical results of this paper. Moreover, we present numerical results for the third example.
\subsection{Distributionally robust sparse convex regression}\label{sec6.1}
The sparse convex regression problem
$$
\min_{\x\in \X} \mathbb{E}[\varphi(\x;\c_\xi, d_\xi)] +\lambda_1\|\x\|_0
$$
has wide applications in data science, where $\X=\{\x\in \R^n:  \underline{\u} \le \x\le \overline{\u}\}$ with $\underline{\u} < \overline{\u}$, $(\c_\xi, d_\xi)\in \R^n\times \R$ represents a random data set of interest, $\varphi(\cdot; \c_\xi, d_\xi)
:\R^n\to \R$ is a convex loss function and $\mathbb{E}$ is the expectation. Widely used convex loss functions include the censored function $(\max(\c_\xi^\top \x,0)-d_\xi)^2$ and the
$\ell_1$ function $|\c_\xi^\top \x-d_\xi|$,  which
are nonsmooth functions. Then, the distributionally robust sparse convex regression problem can be expressed by
\begin{equation}\label{censored1}
\min_{\x\in \X} \max_{\y\in \hat{\Y}} \,  \sum_{i=1}^m \y_i\varphi_i(\x)+\lambda_1\|\x\|_0
\end{equation}
with $\varphi_i(\x):=\varphi(\x;\c_i,d_i)$, a set of $m$ samples $\{\c_i,d_i\}_{i=1}^m$ and the approximation of the ambiguity set
$\hat{\Y}=\{ \y \in \R^{m} : \y\ge 0, \,\, \e^\top \y=1, \,\, \|\textbf{A}\y -\b\|\le {\delta} \}.$
Here $(\textbf{A}, \b, {\delta})\in \R^{k\times m}\times\R^k\times \R_+$ describes the approximated ambiguity set in a general moment form.

Taking account of the constraint on $\y$, the following penalty form
\begin{equation}\label{censored}
\min_{\x\in \X} \max_{\y\in \Y} \,  \underbrace{\sum_{i=1}^m \y_i\varphi_i(\x)-\beta\max\{\|\textbf{A}\y -\b\|^2-{\delta}^2,0\}}_{c(\x,\y)}+\lambda_1\|\x\|_0
\end{equation}
for \eqref{censored1} is promising, where $\beta>0$ is a penalty parameter and $\Y=\{ \y \in \R^{m} : \y\ge 0, \,\, \e^\top \y=1\}$. In \eqref{censored}, $c(\x,\y)$ is nonsmooth with respect to both $\x$ and $\y$. However, thanks to the method in subsection \ref{sec-smoothing}, a smoothing function to $c$ can be easily constructed with the properties in \eqref{s2} and \eqref{s3}. For example, if $\varphi_i(\x)=|\c_i^\top \x-d_i|$, then we can set
$$
\tilde{c}(\x,\y,\varepsilon)=\sum_{i=1}^m \y_i{\theta}(\c_i^\top \x-d_i,{\varepsilon})-\beta{\phi}(\|\textbf{A}\y -\b\|^2-\sigma^2,\varepsilon),
$$
where ${\phi}(s,\varepsilon)$ is a smoothing function of the plus function $s_+$ and $\theta(s,\varepsilon)$ is a smoothing function of the absolute value function $|s|$. Note that ${\phi}(s,\varepsilon)$
can be defined by any one of the following formulations:
$$\phi(s,\varepsilon)=s+\varepsilon\ln(1+e^{-\frac{s}{\varepsilon}}),
\quad \quad
\phi(s,\varepsilon)=\frac{1}{2}(s+\sqrt{s^2+4\varepsilon^2}),
$$
$$
\phi(s,\varepsilon)=\left\{\begin{aligned}
&s_+ &&\mbox{if}\, |s|>\varepsilon\\
&\frac{(s+\varepsilon)^2}{4\varepsilon} &&\mbox{if}\, |s|\leq \varepsilon,
\end{aligned}\right.\quad
\phi(s,{\varepsilon})=\left\{\begin{aligned}
&s +\frac{\varepsilon}{2}e^{-\frac{s}{\varepsilon}} &&\mbox{if}\, s> 0\\
&\frac{\varepsilon}{2}e^{\frac{s}{\varepsilon}} &&\mbox{if}\, s \leq 0,
\end{aligned}\right.
$$
and $\theta(s,\varepsilon)$ can be given by $\theta(s,\varepsilon)=\phi(s,\varepsilon)+\phi(-s,\varepsilon)$.
From Definition \ref{defn2}, it is clear that $\tilde{c}$ is a smoothing convex-concave function of $c$.
Moreover, by Proposition \ref{prop8}, it satisfies \eqref{s2} and \eqref{s3}.
\subsection{Robust bond portfolio construction}
We consider a portfolio of $n$ bonds with quantities $\x\in {\cal X}\subseteq \mathbb{R}^n_+$ and time periods $t=1,\ldots, T$, where the set ${\cal X}=\{\x \in \mathbb{R}^{n} : \underline{\u} \le \x\le \overline{\u}\}$ gives a range of possible quantities for each bond.  Let $\alpha_{i,t}$ denote the cash flow from bond $i$ in period $t$, which includes the coupon payments and the payment of the face value at maturity.

Let $p\in \mathbb{R}^n_+$ denote the price of the bonds with
$$p_i=\sum^T_{t=1}\alpha_{i,t}\text{exp}(-t(u_t+s_i)), \quad i=1,\ldots, n,$$
where $s_i\ge 0$ is the spread for bond $i$ and $u_t$ is the yield curve at time $t$.
The portfolio value is given by $p^\top \x$.
Let $\phi$ be a smooth convex nominal function that may include tracking error against a benchmark, a risk term and a transaction cost term.

Let $\y=(u^\top, s^\top)^\top \in \mathbb{R}^{n+T}$.  The set ${\cal Y}=\{\y \in \mathbb{R}^{n+T} : \underline{\v} \le \y\le \overline{\v}\}$ gives a range of possible values for each point in the yield curve and for each spread.

A version of the robust bond portfolio construction model in \cite{Boyd} is the following convex-concave saddle point problem
$$
\min_{\x\in {\cal X}}\max_{\y\in {\cal Y}} c(\x,\y):= \phi(\x)-\lambda \sum^n_{i=1}\sum^T_{t=1}\x_i\alpha_{i,t}\exp(-t(\y_t+\y_{T+i})) -\beta \|\textbf{A}\y-\b\|_1,$$
where $\|\textbf{A}\y-\b\|_1$ describes the uncertainties in yield curves and spreads, and
$c(\x,\y)$ is a nonsmooth function with respect to $\y$.
A robust bond portfolio construction  with sparse selection of bonds is as follows
 \begin{equation}\label{model1}
\min_{\x\in {\cal X}}\max_{\y\in {\cal Y}} c(\x,\y) +\lambda_1 \|\x\|_0.
\end{equation}

Problem \eqref{model1} is a nonsmooth convex-concave saddle point problem with cardinality penalty $\|\x\|_0$, where
$\X$ is a convex set with int$(\X)\neq\emptyset$, and $\Y$ is a convex set. Note that the assumption int$(\Y)\neq\emptyset$ in (\ref{ass-X}) and Assumption \ref{ass4}-(ii) can be removed, since \eqref{model1} does not have a cardinality function of $\y$. A smoothing function of $\|\textbf{A}\y-\b\|_1$ in the function $c$ can be constructed by $\theta(s,\varepsilon)$ in subsection \ref{sec6.1}.
\subsection{Sparse convex-concave logistic regression min-max problems}\label{sec6.2}
Motivated by the unconstrained convex-concave logistic regression saddle point problem in \cite{Brian}, we consider the following problem
\begin{equation}\label{Brian}
\min_{\x\in\X}\max_{\y\in \Y} \, c(\x,\y):=\sum_{k=1}^N\log(1+e^{-\alpha_k\a_k^\top\x})
+\x^\top {\bf{A}}\y-  \sum_{k=1}^N\log(1+e^{-\beta_k\b_k^\top\y}),
\end{equation}
where $\X=\{\x \in \mathbb{R}^n : {-}\textbf{e}\leq\x\leq\textbf{e}\}$, $\Y=\{\y\in \mathbb{R}^n : { -}\textbf{e}\leq\y\leq{ }\textbf{e}\}$, $\a_k\in \{0,1\}^n, \b_k\in \{0,1\}^m,$ ${\bf{A}}\in\{0,1\}^{n\times m}$ and $\alpha_k, \beta_k \in \{-1,1\}$, for all $k\in [N]$. To find a sparse solution, we consider the following sparse min-max model
\begin{equation}\label{Brian1}
\min_{\x\in\X}\max_{\y\in \Y} \, c(\x,\y) +\lambda_1 \|\x\|_0
-\lambda_2\|\y\|_0.
\end{equation}
It is clear that $c(\cdot,\cdot)$ is a smooth convex-concave function and Assumption \ref{ass4} holds for \eqref{Brian1} with $\tau=\sigma=1$. It has
\begin{align*}
  \nabla_{\x}c(\x,\y)=\sum_{k=1}^N \frac{-\alpha_ke^{-\alpha_k\a_k^\top\x}}{1+e^{-\alpha_k\a_k^\top\x}}\a_k+{\bf{A}}\y,\;
  \nabla_{\y}c(\x,\y)=-\sum_{k=1}^N \frac{-\beta_ke^{-\beta_k\b_k^\top\y}}{1+e^{-\beta_k\b_k^\top\y}}\b_k+{\bf{A}}^\top\x.
\end{align*}
By simple calculation, we can set $L_{c,1}$ defined in \eqref{eq-c} as
\begin{align}\label{Lip}
L_{c,1}=\max\left\{\|\a\|_{\infty}+\|{\bf{A}}\|_{\infty},
\|\b\|_{\infty}+\|{\bf{A}}^\top\|_{\infty}\right\},
\end{align}
where $\a=(\a_1,\ldots,\a_N)$ and $\b=(\b_1,\ldots,\b_N)$.

 If we choose the density function in Example \ref{example4.1} to construct the continuous relaxation function $f^R(\x,\y,\mu)$, then $\alpha=\underline{\rho}=1$.
From the weak d-stationary point defined in Definition \ref{def-d} and by Theorem \ref{theorem4}, if $(\x^*,\y^*)\in\X\times\Y$ is a weak d-stationary point of $\min_{\x\in\X}\max_{\y\in\Y}f^R(\x,\y,\mu)$, then $(\x^*,\y^*)$ is a $\mu$-strong local saddle point of \eqref{Brian1}, that is
\begin{equation}\label{eq-ex-c}
\left\{\begin{aligned}
&|\x_i^*|\not\in(0,\mu),\;|\y_j^*|\not\in(0,\mu),\,&&\forall i\in[n],\;j\in[m]\\
&0\in[\nabla_{\x}c(\x^*,\y^*)]_i+N_{\X_i}(\x_i^*)\;\mbox{if $|\x^*_i|\geq\mu$}, &&\forall i\in[n]\\
&0\in-[\nabla_{\y}c(\x^*,\y^*)]_j+N_{\Y_j}(\y_j^*)\;\mbox{if $|\y^*_j|\geq\mu$}, &&\forall j\in[m].
\end{aligned}\right.
\end{equation}

There are many interesting algorithms for nonconvex-nonconcave min-max problems \cite{Rockafellar1976,Attouch-Wets,Nemirovski2004,Jin-Netrapalli-Jordan2,Grimmer,Xu-Zhang,Niao,Bot2023,Marc}. To illustrate our theoretical results, we solve \eqref{Brian} by the Proximal Gradient Descent Ascent (PGDA) algorithm proposed in \cite{Marc},
 i.e.
\begin{equation}\begin{aligned}
&\x^{k+1}=\arg\min_{\x\in\X}Q(\x,\x^k;\y^k,\y^k)\\
&\y^{k+1}=\arg\max_{\y\in\Y}Q(\x^{k+1},\x^{k+1};\y,\y^k),
\end{aligned}\tag{PGDA}
\end{equation}
where
$$Q(\x,\tilde{\x};\y,\tilde{\y}):=
\langle\nabla_{\x}c(\tilde{\x},\tilde{\y}),\x-\tilde{\x}\rangle
+\langle\nabla_{\y}c(\tilde{\x},\tilde{\y}),\y-\tilde{\y}\rangle+
\frac{1}{2}\gamma\|\x-\tilde{\x}\|^2-\frac{1}{2}\gamma\|\y-\tilde{\y}\|^2,$$
and $\gamma\geq\max\left\{\|\a\|_{\infty},\,\|\b\|_{\infty}\right\}\geq\max_{\x\in\X,\y\in\Y,i\in[n],j\in[m]}\{|[\nabla_{\x\x}^2c(\x,\y)]_{ii}|,
|[\nabla_{\y\y}^2c(\x,\y)]_{jj}|\}$.

To find a weak d-stationary point of the continuous relaxation of \eqref{Brian1}, inspired by the algorithm in \cite{BC-SINUM}, we define
\begin{equation}
Q_{d_{\tilde{\x}},d_{\tilde{\y}}}(\x,\tilde{\x};\y,\tilde{\y};\mu)
:={Q(\x,\tilde{\x};\y,\tilde{\y})}+\lambda_1\sum_{i=1}^n
\Phi^{d_{\tilde{\x}_i}}(\x_i,\mu)-\lambda_2\sum_{j=1}^m
\Phi^{d_{\tilde{\y}_j}}(\y_j,\mu),
\end{equation}
where $\Phi^{d_{\tilde{s}}}(s,\mu)=\left\{\begin{aligned}
&\frac{1}{\mu}|s|&&\mbox{if $|\tilde{s}|<\mu$}\\
&0&&\mbox{if $|\tilde{s}|\geq\mu$}.
\end{aligned}\right.$
Notice that for fixed $\tilde{\x}\in\X$, $\tilde{\y}\in\Y$ and $\mu>0$, $Q_{d_{\tilde{\x}},d_{\tilde{\y}}}(\cdot,\tilde{\x};\cdot,\tilde{\y};\mu)$ is convex-concave.

Combining the PGDA with the alternating index at ${\x}^k$ and $\y^k$ in \cite{BC-SINUM}, we propose the Alternating Proximal Gradient Descent Ascent
(APGDA) algorithm, i.e.
\begin{equation}\label{eq-a-2}
\begin{aligned}
  &\x^{k+1}=\arg\min_{\x\in\X}Q_{d_{{\x}^k},d_{{\y}^k}}(\x,{\x}^k;\y^k,{\y}^k;\mu), \\
  &\y^{k+1}=\arg\max_{\y\in\Y}Q_{d_{{\x}^{k+1}},d_{{\y}^k}}(\x^{k+1},{\x}^k;\y,{\y}^k;\mu).
\end{aligned}\tag{APGDA}\end{equation}
Although the two steps in  APGDA algorithm can be considered as a generalization of Step 2 in Algorithm 3.1 in \cite{BC-SINUM} for solving two strongly convex minimization problems:  $\min_{\x\in\X}f(\x,\y^k)$ and $\min_{\y\in\Y}-f(\x^{k+1},\y)$,
the convergence analysis of APGDA is not trivial. We will study the convergence of APGDA in our future research. Here we only give a remark on PGDA and APGDA.
\begin{remark} By the structure of PGDA and APGDA, we can find the following conclusions directly.
\begin{itemize}
\item Any limit point of $(\x^k,\y^k)$ generated by PGDA is a saddle point of \eqref{Brian}.
\item Any limit point of $(\x^k,\y^k)$ generated by APGDA is a $\mu$-strong local saddle point of \eqref{Brian1}.
\end{itemize}

To determine whether the limit point is a saddle point of \eqref{Brian} or a $\mu$-strong local saddle point of \eqref{Brian1}, by the normal cones of $\X=\{\x:{-}\textbf{e}\leq\x\leq\textbf{e}\}$ and $\Y=\{\y:{ -}\textbf{e}\leq\y\leq{ }\textbf{e}\}$, we define the following evaluation functions
$$R_i(\x)=\left\{
\begin{aligned}
& ([\nabla_{\x}c(\x,\y)]_i)_+&&\mbox{if $\x_i=1$}\\
& (-[\nabla_{\x}c(\x,\y)]_i)_+&&\mbox{if $\x_i=-1$}\\
& |[-\nabla_{\x}c(\x,\y)]_i|&&\mbox{otherwise,}
\end{aligned}
\right.\quad
S_j(\y)=\left\{
\begin{aligned}
& (-[\nabla_{\y}c(\x,\y)]_j)_+&&\mbox{if $\y_j=1$}\\
& ([\nabla_{\y}c(\x,\y)]_j)_+&&\mbox{if $\y_j=-1$}\\
& |[-\nabla_{\y}c(\x,\y)]_j|&&\mbox{otherwise}.
\end{aligned}
\right.
$$
It is clear that $R_i(\x)\geq0$, $\forall i\in[n]$ and $S_j(\y)\geq0$, $\forall j\in[m]$. For $\bar{\x}\in\X$ and $\bar{\y}\in\Y$, it holds
\begin{itemize}
\item $p(\bar{\x}):=\sum_{i=1}^nR_i(\bar{\x})=0$ and $q(\bar{\y}):=\sum_{j=1}^mS_j(\bar{\y})=0$ if and only if $(\bar{\x},\bar{\y})$ is a saddle point of \eqref{Brian}.
\item $\tilde{p}(\bar{\x}):=\sum_{i:\bar{\x}_i\neq0}(R_i(\bar{\x})+\max\{\mu-|\bar{\x}_i|,0\})=0$ and $\tilde{q}(\bar{\y}):=\sum_{j:\bar{\y}_j\neq0}(S_j(\bar{\y})+\max\{\mu-|\bar{\y}_j|,0\})=0$ if and only if $(\bar{\x},\bar{\y})$ is a $\mu$-strong local saddle point of \eqref{Brian1}.
\end{itemize}
\end{remark}

To demonstrate the sequence convergence graphically, we conduct a simple test experiment. We choose $n =20$, $m=30$, $N=50$, $\lambda_1=\lambda_2=1$, randomly  generate ${\bf{A}}\in\{0,1\}^{n\times m}$, and for $k\in[N]$ randomly generate  $\a_k\in\{0,1\}^n$, $\b_k\in\{0,1\}^m$ with $2$ nonzero elements, $\alpha_k, \beta_k\in \{-1,1\}$. We compute the constant $L_{c,1}$ as in \eqref{Lip} and obtain $\mu=0.0323<\bar{\mu}_1$,
where $\bar{\mu}_1=\min\left\{1,1/{L_{c,1}}
\right\}$ is defined as in \eqref{eq-mu}.  We choose an initial point $\x^0=0.2{\bf e},\y^0=0.2{\bf e}$ for running both PGDA and APGDA.

 Fig. \ref{fig6} plots the convergence of $\x^k$, $\y^k$, $p(\x^k)$,
$q(\y^k)$ generated by PGDA and convergence of $\x^k$, $\y^k$, $\tilde{p}(\x^k)$,
$\tilde{q}(\y^k)$ generated by APGDA.
From Fig. \ref{fig6}, we find that the limit point of the sequence $(\x^k,\y^k)$ generated by PGDA has no zero element and some elements of it do not satisfy the lower bounds in \eqref{eq-ex-c}. However,
more than half elements of the limit point of $(\x^k,\y^k)$ generated by APGDA are zero, and all elements of it satisfy the lower bounds in \eqref{eq-ex-c}.
This is consistent with the theoretical results and shows the superiority of  \eqref{Brian1} in finding a sparse solution. Moreover,
in Fig. \ref{fig6}-(e), from the convergence of $p({\x}^k)$ and $q({\y}^k)$ on $(\x^k,\y^k)$ generated by PGDA, we confirm that the limit point of $(\x^k,\y^k)$ is a saddle point of \eqref{Brian}, while Fig. \ref{fig6}-(f) shows the convergence of $\tilde{p}({\x}^k)$ and $\tilde{q}({\y}^k)$ on $(\x^k,\y^k)$ generated by APGDA, which confirms that the limit point of this sequence is a $\mu$-strong local saddle point of \eqref{Brian1}.
\begin{figure}[htp]
\begin{center}
  \subfigure[$\x^k$ by PGDA]{
    \includegraphics[width=2in,height=1.8in]{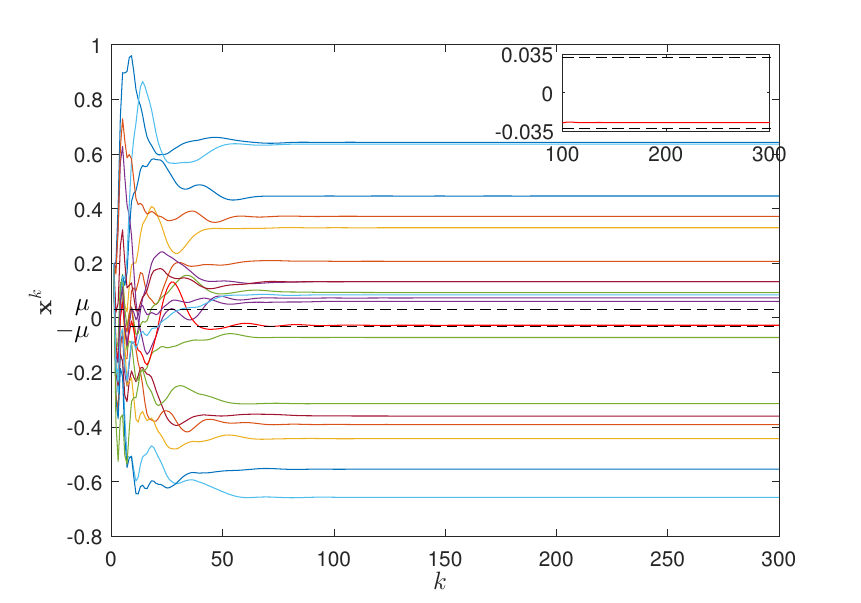}}
     \subfigure[$\x^k$ by APGDA]{
    \includegraphics[width=2in,height=1.8in]{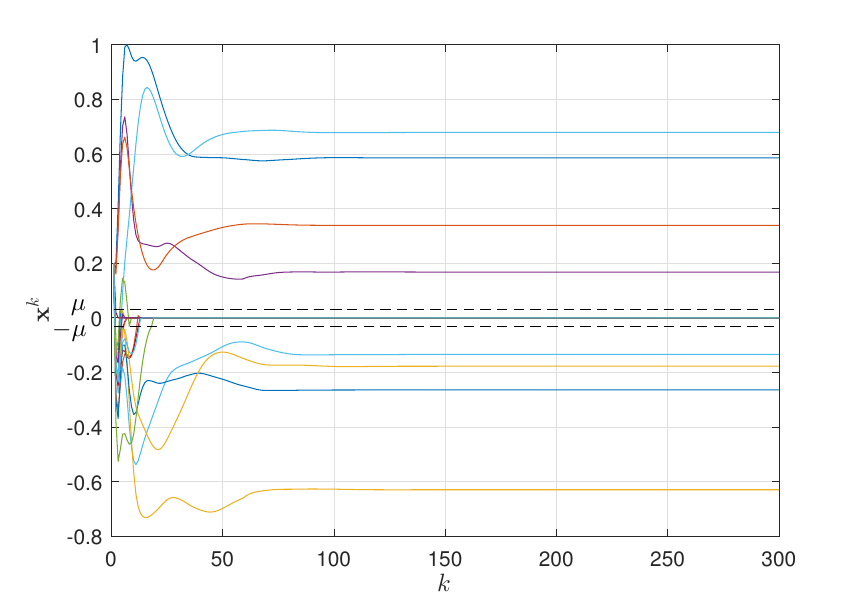}}
    \subfigure[$\y^k$ by PGDA]{
    \includegraphics[width=2in,height=1.8in]{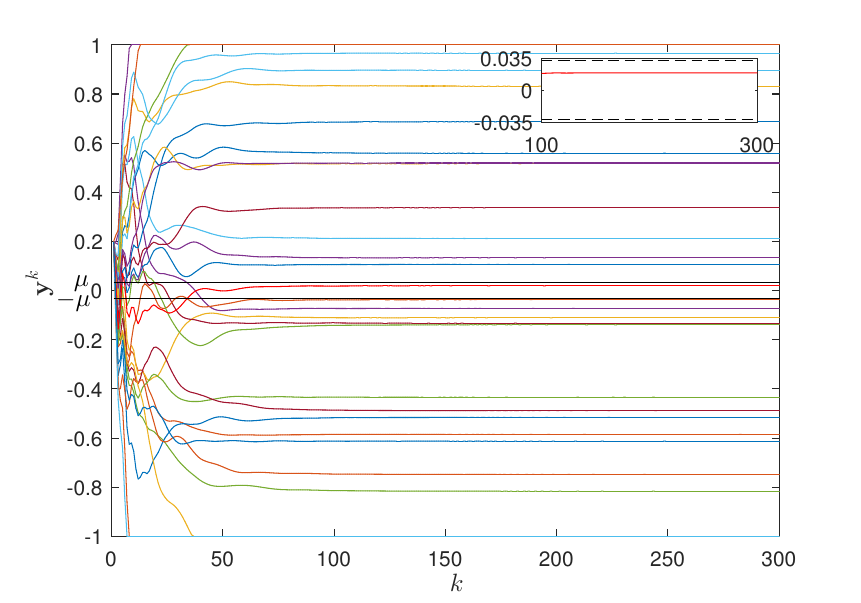}}
     \subfigure[$\y^k$ by APGDA]{
    \includegraphics[width=2in,height=1.8in]{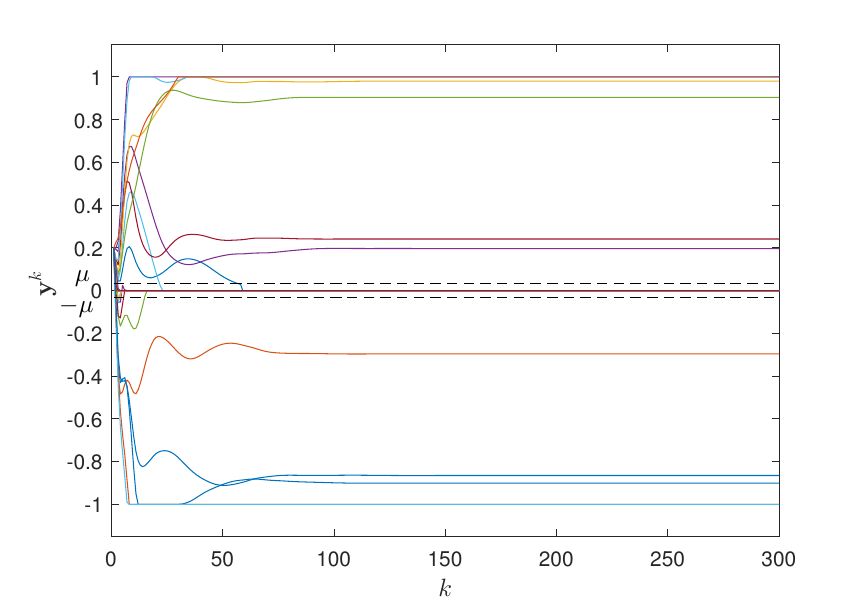}}
  \subfigure[$p({\x}^k)$ and $q({\y}^k)$ by PGDA]{
    \includegraphics[width=2in,height=1.8in]{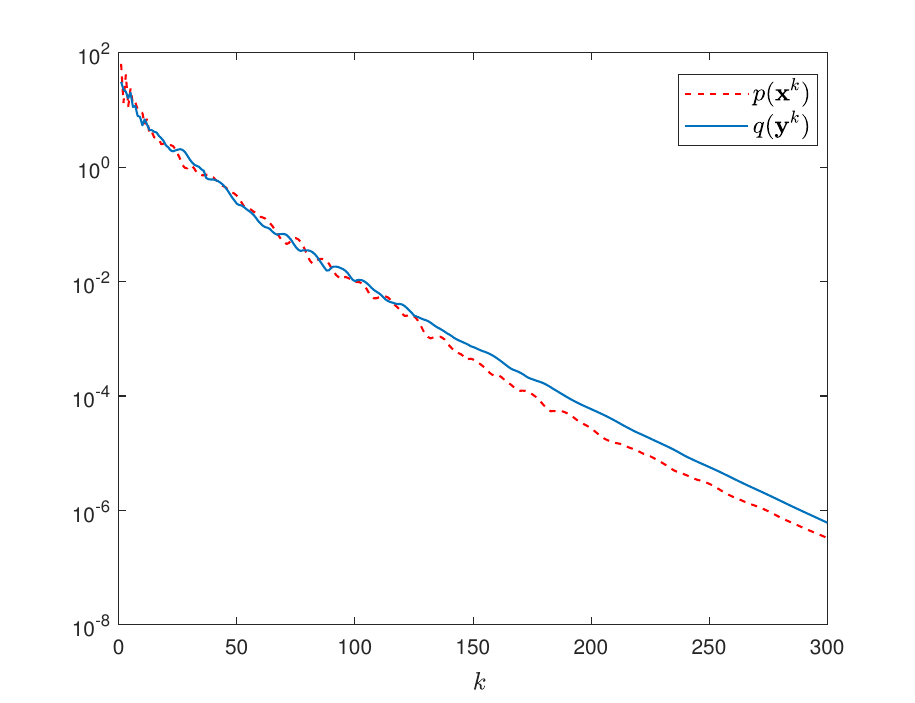}}
  \subfigure[$\tilde{p}({\x}^k)$ and $\tilde{q}({\y}^k)$ by APGDA]{
    \includegraphics[width=2in,height=1.8in]{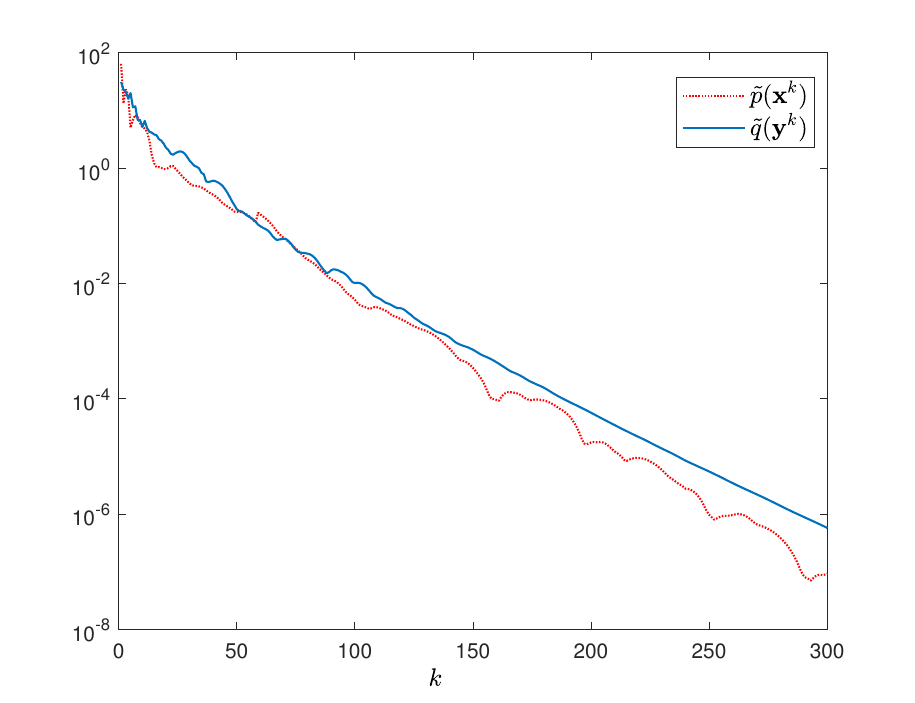}}
  \caption{Convergence of $\x^k$, $\y^k$, $p({\x}^k)$ and $q({\y}^k)$ generated by PGDA and convergence of  $\x^k$, $\y^k$, $\tilde{p}({\x}^k)$ and $\tilde{q}({\y}^k)$ generated  by APGDA }\label{fig6}
\end{center}\end{figure}

\section{Conclusion}
In this paper,  we prove the existence of local saddle points and global minimax points of problem \eqref{obb} and define a class of strong local saddle points of it. To construct interesting continuous relaxations to \eqref{obb} based on the convolution,
 we introduce two classes of  density functions which satisfy Assumptions \ref{ass7} and \ref{ass8}, respectively. The induced continuous relaxations include the capped-$\ell_1$, scaled SCAD, scaled MCP, hard thresholding functions as special cases.  Moreover, we establish the relations between problem \eqref{obb} and its continuous relaxation (\ref{ob-r}) regarding their saddle points, local saddle points  and global minimax points by using the lower bound properties of $g(\x)$ and $h(\y)$ in \eqref{eq-ass6-1}-\eqref{eq-ass6-2} at the local saddle points and global minimax points of the continuous relaxation problem.  Moreover, we define the weak d-stationary points and weak second order stationary points of problem (\ref{obb-3-1}), which are necessary conditions for the local saddle points of its continuous relaxation problem \eqref{obb-3r}, while sufficient conditions for the strong local saddle points of (\ref{obb-3-1}). In addition, we study smoothing approximation of \eqref{obb-3r} by using a smoothing convex-concave function of nonsmooth $c$ and prove that any accumulation point of weak d-stationary points of the smoothing approximation problem is a weak d-stationary point of  \eqref{obb-3r} as the smoothing parameter goes to zero.


\end{document}